\documentclass[a4paper,11pt,reqno]{amsart}

\usepackage{amsmath,multicol} 
\usepackage[all]{xy}

\newtheorem{theorem}{\bf Theorem}[subsection]
\newtheorem{prop}[theorem]{\bf Proposition}

\newtheorem{lemma}[theorem]{\bf Lemma}

\newtheorem{definition}[theorem]{\bf Definition}

\theoremstyle{remark}
\newtheorem{example}[theorem]{\bf Example}

\theoremstyle{remark}
\newtheorem{rem}[theorem]{\bf Remark}

 \numberwithin{equation}{subsection}

\newcommand{\pointir}{\unskip . ---\ignorespaces\,}

\newcommand{\ie}{{\it i.e.\;}}

\newcommand{\ad}{\operatorname{ad}}

\def\go{\mathfrak}
\def\bb{\mathbb}
\def\C{\mathbb C}
\def\cal{\mathcal}

 \def\adots{\mathinner{\mkern2mu\raise1pt\hbox{.}
\mkern3mu\raise4pt\hbox{.}\mkern1mu\raise7pt\hbox{.}}}
 
 \title[Decomposition of  regular Prehomogeneous Vector Spaces]{Decomposition of reductive regular Prehomogeneous Vector Spaces}

\author{Hubert Rubenthaler}

\address
{Hubert Rubenthaler\\ Institut de Recherche Math\'ematique Avanc\'ee\\
Universit\'e de Strasbourg et CNRS\\
7 rue Ren\'e Descartes\\
67084 Strasbourg Cedex\\ France\\
E-mail: {\tt rubenth@math.unistra.fr}}

\begin{document}
 
\parindent=0pt

 \maketitle

\begin{abstract} Let $(G,V)$ be a regular prehomogeneous vector space (abbreviated  to $PV$), where $G$ is a connected reductive algebraic group over $\bb C$. If $V= \oplus_{i=0}^{n}V_{i}$ is a decomposition of $V$ into irreducible representations, then, in general, the PV's  $(G,V_{i})$ are no longer regular. In this paper  we introduce  the notion of quasi-irreducible $PV$ (abbreviated to  $Q$-irreducible), and show first that for completely $Q$-reducible $PV$'s, the $Q$-isotopic components are intrinsically defined, as in ordinary representation theory. We also show that, in an appropriate sense, any regular PV  is a direct sum of quasi-irreducible $PV$'s. Finally we classify the quasi-irreducible PV's of parabolic type.
  \end{abstract}
 
\maketitle

\medskip\medskip\medskip\medskip\medskip\medskip

 \section{Introduction}
 \vskip 10pt
 \subsection{}
Let us first recall that a prehomogeneous vector space (abbreviated to $PV$)  is a triplet  $(G,\rho, V)$ where $G$ is an algebraic group over $\bb C$, and $\rho$ is a rational representation of $G$ on the finite dimensional vector space $V$, such that $G$ has a Zariski open orbit in $V$.     The theory of PV's was created by Mikio Sato in the early 70's to provide generalizations of several kinds of known local or global zeta functions satisfying a  functional equation similar to that of the Mellin transform, the Riemann zeta function,  the Epstein zeta function or the zeta function of a simple algebra \cite{Sato-M}.

For the basic results on PV's we refer the reader to \cite{Sato-Kimura} and to \cite{Kimura}.

There are   many papers concerned with  local or global zeta functions of PV's and their functional equations. Among them 
 let us mention \cite{Sato-Shintani},  \cite{F-Sato-3}, \cite{Saito}, \cite{Bopp-Rubenthaler-AMS},\cite {F-Sato-4} for example.
 
   There are also many papers concerning the classification theory of $PV$'s. Many of them are written by T. Kimura and his students. We refer to the bibliography of \cite{Kimura} and to \cite{Kimura-2}, \cite{Kimura-3}, \cite{Kimura-4}, \cite{Kimura-5} for more details.  The  regular $PV$'s of parabolic type were classified in  \cite{Mortajine}.
 

         \vskip 10pt
\subsection{} In order to associate a zeta function to a  reductive $PV$ one needs a further condition on the $PV$, namely the so-called  {\it regularity}  condition (see section  2.1) Therefore  knowledge of the structure of the reductive regular $PV$'s as well as their classification is of particular interest. Unfortunately if $(G,V)$ is a non irreducible reductive regular $PV$, it can be seen in  easy examples (see example \ref {exemple1}) that the irreducible components of the representation $(G,V)$, which are still prehomogeneous, are in general not regular.   This makes    understanding  the structure of such $PV$'s   difficult.  To get around this difficulty we introduce the notion of quasi-irreducible $PV$ (abbreviated to $Q$-irreducible) and show that, in an appropriate sense,  any reductive regular $PV$ is a sum of $Q$-irreducible $PV$'s. 

\vskip 10pt
\subsection{}  

 Let us now describe the content of the paper.

  It is worthwhile pointing out  that usually the group $G$ of a $PV $ is supposed to be connected. For our purpose we do not make this hypothesis. 
Therefore in section 2.1 we begin by giving  extensions of basic results to the case where the group is not connected. 

In section 2.2 we  give the definition of $Q$-irreducible $PV$'s and  prove that, if $G$ is reductive and if  $(G,V)$ is a regular PV which is completely Q-reducible, then  the Q-isotypic components of $(G,V)$ are intrinsically  defined.

 In section 3 we  give our structure theorem for reductive regular $PV$'s which asserts  that, if $(G,V)$ is a regular reductive $PV$,  there exists a filtration of the space $V$:

$$\{0\}=U_{k+1}\subset U_{k}\subset\dots\subset U_{2}\subset U_{1}=V,$$ and  a filtration of the group $G$
$$  G_{k}\subset G_{k-1}\subset\dots\subset G_{1}=G,$$
such that  the $G_{i}'s$ are reductive and the $U_{i}$ and $U_{i+1}$ are $G_{i}$-stable. Moreover $(G_{i}, U_{i})$ is a regular PV and $(G_{i},U_{i}/U_{i+1})$ is completely $Q$-reducible, for $i=1,\dots,k$. See Theorem \ref{th-decomposition}  below for the precise statement.

In section 4.1 we  give a brief account of the theory of parabolic $PV$'s, and in section 4.2  we  give the complete classification of regular $Q$-irreducible $PV$'s.

\subsection{Acknowledgement} I obtained the results of this paper a long time ago, but  never published them. I would like to thank Tatsuo Kimura for  the recent stimulating conversations about  classification theory of $PV$'s which convinced me to write them up.

\section{Completely $Q$-reducible regular PV's}

\subsection{The regularity for non connected reductive groups}
 As said in the Introduction a prehomogeneous vector space    is a triplet  $(G,\rho, V)$ where $G$ is an algebraic group over $\bb C$, and $\rho$ is a rational representation of $G$ on the finite dimensional vector space $V$, such that $G$ has a Zariski open orbit in $V$.   The open orbit is usually denoted by $\Omega$   and $S=V\setminus \Omega$ is the {\it singular set}. The elements in the open orbit are called ${\it generic}$. We often simply write $(G,V)$ for a $PV$ when we do not need to make the representation explicit. A {\it relative invariant} of the $PV$ $(G,V)$ is a rational function $f$ on $V$, such that there exists a rational character $\chi $ of $G$, such that for all $x\in \Omega$ and all $g \in G$, one has $f(g.x)=\chi(g)f(x)$. The character $\chi$ determines $f$ up to a multiplicative constant. The subgroups we shall consider in the sequel are isotropy subgroups. These will be reductive, but not necessarily connected. Therefore we need to extend slightly the basic results concerning the  regularity.

\begin{prop}\label{prop.PV-non-connexe} Let $(G,V)$ be a PV, where $G$ is not necessarily connected and not necessarily reductive.  Let $G^\circ$ be the connected component group of $G$. Denote by $\Omega$ the open orbit under $G^\circ$ and define $S=V\setminus \Omega$.  Let $S_{1},\dots, S_{k}$ be the irreducible components of codimension one in $S$. Let $f_{1},f_{2},\dots,f_{k}$ be irreducible polynomials such that
$$S_{i}=\{x\in V | f_{i}(x)=0\}.$$
The $f_{i}$'s are (as well known) the fundamental relative invariants of $(G^\circ, V)$. 

Then:

{\rm1)} $\Omega$ is also the open $G$-orbit.

{\rm 2)} For any $g\in G$ and for any $i \in \{1,\dots,k\}$, there exists $\sigma^g(i) \in \{1,\dots,k\}$ and a non zero contant $c(i,g)$ such that  for all $x\in V$, one has $ f_{i}(g.x)= c(i,g)f_{\sigma^g(i)}(x)$. Therefore the group $G$ acts by permutations on the set of lines  $\{{\bb C}f_{i}, i=1,\dots,k\}$.

{\rm 3)} Let $I_{1}\cup I_{2}\cup \dots\cup I_{r}=\{1,2,\dots,k\}$ be the partition defined by the $G$-action on the lines  ${\bb C}f_{i}$. Define $\varphi_{j}=\prod_{i\in I_{j}}f_{i}$. Then $\varphi_{j}$ is a relative invariant under $G$. Any relative invariant $\varphi$ under $G$ can be uniquely written in the following way:
$$\varphi= c \varphi_{1}^{m_{1}}\varphi_{2}^{m_{2}}\dots\varphi_{r}^{m_{r}}$$
where $m_{j}\in {\bb Z}$ and $c\in {\bb C}$.

\end{prop}
\begin{proof}\hfill 

1) Let $\Omega$ be the open $G^\circ$-orbit of $V$. Let us  prove first that for any $g\in G$ the set  $g.\Omega $ is a $G^\circ$-orbit. Let $u=g.x$ and $v=g.y$ ($x,y \in \Omega$) be two elements in $g.\Omega$. By definition there exists $h\in G^\circ$ such that $x=h.y$. Therefore
$$u=g.x=gh.y=ghg^{-1}g.y= h'g.y=h'.v $$
  (where  $ h'= ghg^{-1}  \in  G^\circ$,   because $G^\circ$  is  a normal subgroup of $G$). As $g.\Omega$ is open, we have $g.\Omega=\Omega$, for all $g\in G$. Hence $\Omega  $ is also the open $G$-orbit.
  
  2) Denote by $\chi_{i}$ the $G^\circ$ character of $f_{i}$. For $g\in G$ and $x\in V$, define $\psi_{i}^g(x)=f_{i}(g.x)$. Then for $h\in G^\circ$ we have $\psi_{i}^g(h.x)=f_{i}(gh.x)=f_{i}(ghg^{-1}g.x)=\chi_{i}(ghg^{-1}) \psi_{i}^g(x)$. Therefore $\psi_{i}^g$ is an irreducible relative invariant of $G^\circ$. Hence there exists $\sigma^g(i) \in \{1,\dots,k\}$ and a non zero contant $c(i,g)$ such that  for all $x\in V$, one has $\varphi_{i}(x)= f_{i}(g.x)= c(i,g)f_{\sigma^g(i)}(x)$ .   
  
  3)   Let $\varphi_{j}$ as defined above. Let $g\in G$. One has $\varphi_{j}(g.x)= \prod_{i\in I_{j}}f_{i}(g.x)=(\prod_{i\in I_{j}}c(i,g))\varphi_{j}(x)$. Hence $\varphi_{j}$ is a relative invariant under $G$, with character $\widetilde{ \chi}_{j}(g)=(\prod_{i\in I_{j}}c(i,g))$. Let $\varphi$ be a relative invariant under $G$. Let $\chi_{\varphi}$  be the corresponding $G$ character. As $\varphi$ is a relative invariant under $G^\circ$, one has $\varphi=c\prod _{i=1}^k f_{i}^{n_{i}}$, where $c\in {\bb C}$ and where $n_{i} \in {\bb Z}$.
  We  have, for $g\in G$ and $x\in \Omega$:
  $$\varphi(g.x)=c\chi_{\varphi}(g)\prod_{i=1}^k f_{i}^{n_{i}}(x)= c\prod_{i=1}^k f_{i}^{n_{i}}(g.x) =c'\prod _{i=1}^kf_{\sigma^g(i)}^{n_{i}}(x)  \quad (c' \in {\bb C})$$
  Therefore from the uniqueness  of the decomposition for $G^\circ$ relative invariants, we obtain that for every $g\in G$ we have $n_{\sigma^g(i)}=n_{i}$. Hence the powers $n_{i}$ of the $f_{i}$'s in the same subset $I_{j}$, are the same, say  $m_{j}$. This implies that
  $$\varphi= c \varphi_{1}^{m_{1}}\varphi_{2}^{m_{2}}\dots\varphi_{r}^{m_{r}}.$$

\end{proof}
\begin{rem}

Of course all the $f_{i}$ where $i\in I_{j}$ have the same degree.
\end{rem}

\begin{definition} Let $(G, V)$  be a PV where $G$ is a reductive, non necessarily connected, algebraic group. The PV $(G,V)$ is called regular if there exists a relative invariant $f$ such that $  \frac{df}{f}=gradlog(f): \Omega \longrightarrow V^*$ is generically surjective $($i.e. has a Zariski dense image$)$. Such a relative invariant is said to be {\it nondegenerate}.
\end{definition} 

\begin{prop}  \label{prop.PV-regulier-non-connexe}{\rm (Compare with \cite{Kimura}, Th. 2.28), and \cite{Sato-Kimura}, Remark 11 p.64)}\hfill  

Let $G$ be a reductive algebraic group. Let $G^\circ$ be the connected component group of $G$ and suppose that $(G,V)$ is a  PV.  
 
 The following conditions are equivalent:

{\rm i)} $(G,V)$ is regular.

{\rm ii)} There exists a relative invariant $f$ such that   the   Hessian $H_{f}(x)=\text{Det}(\frac{\partial^2 f }{\partial x_{i}\partial x_{j}}(x) )$ is not identically zero

{\rm iii)}  The singular set  $S$ is a hypersurface.

{\rm iv)} The open orbit $\Omega= V\setminus S$ is an affine variety.

{\rm v)} Each generic isotropy subgroup is reductive.

{\rm vi)} Each generic isotropy subalgebra is reductive.

Suppose moreover that these conditions hold. Then any polynomial $f$ satisfying  $S=\{x\in V| f(x)=0\}$ is a nondegenerate relative invariant of $G^\circ$. In the notations of Proposition \ref{prop.PV-non-connexe}  the set of such   polynomials which are relative invariants under $G$ is the set of polynomials of the form
$$f= c\varphi_{1}^{m_{1}}\varphi_{2}^{m_{2}}\dots\varphi_{r}^{m_{r}}$$
where $m_{j}\in {\bb N}^*$ and $c\in {\bb C}^*$.

\end{prop}

\begin{proof}\hfill 

  We will of course use the same result for connected reductive groups (\cite{Kimura}, Th. 2.28. and \cite{Sato-Kimura})
  
  First of all we remark that by the same proof as in the case where the group is connected (see \cite{Sato-Kimura}, Proposition 10 p.62 and Remark 11 p. 64) we obtain ${\rm i) \Leftrightarrow {\rm ii)}}$.
 
  i) $\Rightarrow$ iii):  If $(G,V)$ is regular, there exists a nondegenerate relative invariant $f$. This function is also a relative invariant of $(G^\circ, V)$, hence the singular set for the $G^\circ$ action is an hypersurface. But the singular set for $G$ is the same as for $G^\circ$, from  Proposition \ref {prop.PV-non-connexe}. Assertion iii) is proved.
  
  iii) $\Rightarrow$ iv): This is classical: the complementary set of a hypersurface is allways an affine variety.
  
   iv) $\Rightarrow$ v):  From \cite{Kimura}, Th. 2.28, we know that for $x\in \Omega$, the isotropy subgroup $G^\circ_{x}$ is reductive. Hence the isotropy subgroup $G_{x}$ is reductive.
   
    v) $\Rightarrow$ vi): As  the Lie algebras of $G^\circ$ and of $G$ are the same, this is obvious.
    
     vi) $\Rightarrow$ i): Let $S_{1}, \dots, S_{m}$ be the irreducible components of $S$. They correspond to irreducible polynomials $f_{1},\dots,f_{m}$    which are the fundamental relative invariants for $G^\circ$. We know from   \cite{Kimura}, Th. 2.28  that if vi) is satisfied then $(G^\circ,V)$ is regular and therefore any polynomial $f$  such that $S=\{x\in V| f(x)=0\}$ is a nondegenerate relative invariant under $G^\circ$. Among them the functions which are relative invariants under  $G$ are of the proposed form from Proposition  \ref{prop.PV-non-connexe}.  Hence $(G,V)$ is regular and i) is true.

    \end{proof}

  \begin{rem}
 Under the assumptions of the preceding Proposition, the polynomial $f=f_{1}f_{2}\dots f_{k}=\varphi_{1}\varphi_{2} \dots \varphi_{r}$ is the unique polynomial of minimal degree which defines  $S$. It is a relative invariant under $G$.
  
  \end{rem}

\subsection{Quasi-irreducible PV's and complete Q-reducibility}\hfill 

 \vskip 3pt

The following result is often very useful.
\begin{prop}\label {prop-components} Let $(G,V)$ be a PV. Here we do not  suppose that $G$ is connected and we do not  suppose that $G$ is reductive. Suppose that $V=V_{1}\oplus V_{2}$ where $V_{1}$ and $V_{2}$ are two non trivial $G$-invariant subspaces of $V$. Denote by $p_{1}$ $($resp. $p_{2}$$)$ the projections on $V_{1}$ $($resp. $V_{2}$$)$ defined by this decomposition.

{\rm 1)} The representations $(G,V_{1})$ and $(G,V_{2})$ are PV's.  Moreover the open orbits in $V_{1}$ (resp. $ V_{2}$) are the projections of $\Omega_{V}$ i.e. $\Omega_{V_{i}}=p_{i}(\Omega_{V}), i=1,2$.

{\rm 2)} Let $x_{0}+y_{0}$ be a generic element of $(G,V)$, with $x_{0}\in V_{1}$ and $y_{0 }\in V_{2}$. Let also $G_{x_{0}}$ $($resp. $G_{y_{0}}$$)$ be the isotropy subgroup of $x_{0}$ $($resp. $y_{0}$$)$. Then $(G_{y_{0}}, V_{1})$ and $(G_{x_{0}}, V_{2})$ are PV's, and $x_{0}$ is generic in $(G_{y_{0}}, V_{1})$ and $y_{0}$ is generic in $(G_{x_{0}}, V_{2})$.

{\rm 3)} One has $G_{x_{0}}\cap G_{y_{0}}=G_{x_{0}+y_{0}}$.  The open $G_{y_{0}}$-orbit in $V_{1}$ is equal to $\Omega_{V_{1}}(y_{0})=\{x\in V_{1}, x+y_{0}\in \Omega_V\}$ and the open $G_{x_{0}}$-orbit in $V_{2} $ is equal to  $\Omega_{V_{2}}(x_{0})= \{y\in V_{2}, x_{0}+y \in \Omega_V\}$.

{\rm 4)} The subgroup $\widetilde G$  generated by $G_{x_{0}}$ and $G_{y_{0}}$ is open, and hence closed, therefore we he have $\widetilde G=G$ if $G$ is connected. More precisely the subset  $G_{x_{0}}.G_{y_{0}}$ is open in $G$.

{\rm 5)} Suppose that $G$ is reductive, and that $(G,V)$ and $(G,V_{1})$ are regular. Then $(G_{x_{0}}, V_{2})$ is a regular reductive $PV$.

\end{prop}

\begin{proof}\hfill 
 
1) As the projections $p_{1}$ and $p_{2}$ are open maps, the sets $\Omega_{V_{i}}=p_{i}(\Omega_{V}), i=1,2$ are open. Let $x_{1}$ and $x_{2}$ be two elements in $\Omega_{V_{1}}$. From the definition there exists $y_{1}$ and $y_{2}$ in $V_{2}$ such that $x_{1}+y_{1}$ and $x_{2}+y_{2}$ belong to $\Omega_{V}$. Therefore there exists $g\in G$ such that $g.(x_{1}+y_{1})=x_{2}+y_{2}$. Hence $g.x_{1}=x_{2}$. Hence two elements in $\Omega_{V_{1}}$ are congugate. Conversely the conjugate of an element in $\Omega_{V_{1}}$ is still in $\Omega_{V_{1}}$. This proves the first assertion for $V_{1}$. The argument for $V_{2}$ is the same.

2) Define $n=\dim V$, $n_{1}=\dim V_{1}$, $n_{2}=\dim V_{2}$. As $(G,V)$ is prehomogeneous, we have $n=\dim G - \dim G_{x_{0}+y_{0}}$ and as $(G,V_{1})$ is also prehomogeneous we have $n_{1}=\dim G -\dim G_{x_{0}}$. Therefore
 $$\begin{array}{lll}
 n=n_{1}+n_{2}&=& \dim G-\dim G_{x_{0}+y_{0}}\\ 
{}&=&\dim G-\dim G_{x_{0}}+\dim G_{x_{0}}-\dim G_{x_{0}+y_{0}}\\
&=&  n_{1}+\dim G_{x_{0}}-\dim G_{x_{0}+y_{0}}.
\end{array}$$

Therefore $n_{2}=\dim G_{x_{0}}-\dim G_{x_{0}+y_{0}}$ and as $G_{x_{0}+y_{0}}=(G_{x_{0}})_{y_{0}}$ is the isotropy subgroup of $y_{0}$ in $G_{x_{0}}$, the representation $(G_{x_{0}}, V_{2})$ is prehomogeneous, and $y_{0}$ is generic for this space. 

3)  The assertion $G_{x_{0}}\cap G_{y_{0}}=G_{x_{0}+y_{0}}$ is obvious. It is clear that $\Omega_{V_{2}}(x_{0})= \{y\in V_{2}, x_{0}+y \in \Omega_V\}$ is stable under $G_{x_{0}}$. Moreover if  $y_{1}, y_{2 }\in \Omega_{V_{2}}(x_{0})$, then $x_{0}+y_{1}, x_{0}+y_{2}\in \Omega_{V}$ and there exists $g\in G$ such that $g(x_{0}+y_{1})=x_{0}+y_{2}$, and hence $g\in G_{x_{0}}$ and $g y_{1}=y_{2}$. This proves that the open $G_{x_{0}}$-orbit in $V_{2}$ is $\Omega_{V_{2}}(x_{0)}$.  The proof for the space $(G_{y_{0}},V_{2})$ is symmetric.

4) Consider the set ${\cal O}=( \Omega_{V_{1}}(y_{0})\oplus \Omega_{V_{2}}(x_{0})\cap \Omega_{V}$. This set is nonempty ($x_{0}+y_{0}\in {\cal O})$ and open. Let $x+y\in {\cal O}$. Then $x\in \Omega_{V_{1}}(y_{0})$ and we know from the third assertion that there exists $g_{1}\in G_{y_{0}}$ such that $g_{1}x=x_{0}$. Hence $g_{1}(x+y)=x_{0}+g_{1}y$. As $x+y \in \Omega_{V}$, we have also $x_{0}+g_{1}y\in \Omega_{V}$. Hence $g_{1}y\in \Omega_{V_{2}}(x_{0})$. Then we know that there exists $g_{2}\in G_{x_{0}}$ such that $g_{2}g_{1}y=y_{0}$. Hence $g_{2}g_{1}(x+y)=x_{0}+y_{0}$. Therefore the elements of ${\cal O}$ are conjugate under the set $G_{x_{0}}.G_{y_{0}}$. Hence $G_{x_{0}}.G_{y_{0}}/G_{x_{0}+y_{0}}\simeq {\cal O}$ is an open subset of $G/G_{x_{0}+y_{0}}\simeq \Omega$. This implies that $G_{x_{0}}.G_{y_{0}}$ is open in $G$. Therefore the group $\widetilde G$ generated by $G_{x_{0}}$ and $G_{y_{0}}$  is open and hence closed. If G is connected, then $\widetilde G= G$

5) From Proposition \ref{prop.PV-regulier-non-connexe} we know that $G_{x_{0}}$ is reductive and from Proposition \ref{prop-components} we know that $(G_{x_{0}},V_{2})$ is a $PV$. As $(G_{x_{0}})_{y_{0}}= G_{x_{0}+y_{0}}$, using again Proposition \ref {prop.PV-regulier-non-connexe},  we obtain that $(G_{x_{0}},V_{2})$ is regular.

\end{proof}
Unfortunately the irreducible components of a reductive regular $PV$ are in general not regular as shown by the following example.
\begin{example}\label{exemple1}\hfill{}

\noindent Let $G= {\bb C}^*\times SL_{n}\times {\bb C}^*$, let $V={\bb C}^n\times {\bb C}^n$ and define $\rho$ as follows:
$$\rho(x,g,y)(v,w)=(x{^t}vg^{-1},y^{-1}gw)$$
where $x,y \in {\bb C}*$, $g\in SL_{n}$, where $v,w \in {\bb C}^n$ are considered as column matrices and where ${^t}v$  is the transpose of the vector $v$.    A simple computation shows that if $v_{0}=w_{0}= {^t}(1,0,\dots,0)$ then the isotropy subgroup is the set of triplets $(x,\begin{pmatrix}x&0\\
0&A
\end{pmatrix}, x)$, where $A\in GL_{n-1}$, and such that $x. {\text {Det}}A=1$, and this proves that $(G,\rho,V)$ is a regular PV. In fact the scalar product $Q(v,w)={^t}v.w$ of $v$ and $w$ is the unique relative invariant. The irreducible components are $V_{1}={\bb C}^n\times \{0\}$ and $V_{2}=\{0\}\times {\bb C}^n$, and the PV's $(G, \rho_{|_{V_{i}}},V_{i})$ ($i=1,2$) are obviously not regular.
\end{example}
\vskip 5pt

The following lemma is also useful in the sequel.
\begin{lemma}\label{lemma-extension-restr-invariants}
Let $(G,V)$ be a $PV$ where $G$ is not necessarily connected and not  necessarily reductive and suppose that $V=V_{1}\oplus V_{2}$ where $V_{1}$ and $V_{2}$ are $G$-invariant subspaces.  

{\rm a)} Let $f$ be a relative invariant of $(G,V_{1})$. Then the function $\tilde f$ defined by $\tilde f (x+y)=f(x)$ $(x\in V_{1}, Y\in V_{2})$ is a relative invariant of $(G,V)$ with the same character as $f$.

{\rm b)} Let $f$ be a relative invariant of $(G,V)$ which is defined and nonzero on an open subset of $V_{1}$, then for $x\in V_{1}, y\in V_{2} $, we have $f(x+y)=f(x)$.
\end{lemma}

\begin{proof}\hfill 

a) Let $\chi_{f}$ be the character of $f$. For $g\in G$, we have: 
$$\tilde f(g.x+g.y)=f(g.x)=\chi_{f}(g)f(x)=\chi_{f}(g)\tilde f(x+y).$$

b) For $x\in V_{1}, y\in V_{2} $ let us set $\tilde f (x+y)=f(x)$. From a) we know that $\tilde f$ is a relative invariant of $(G,V)$ with the character as $f$. Therefore there exists a constant $c\in {\bb C}$ such that $\tilde f= c.f$. But as $\tilde f=f$ on $U$, we   have necessarily $c=1$

\end{proof}

\begin{definition}\label{def-irreductibles}

Let $G$ be a    reductive group  $($not necessarily connected$)$  and let $(G,V) $ be a regular $PV$.

{\rm a)}The prehomogeneous vector space $(G,V)$ is called $1$-irreducible if the singular set $S=V\setminus \Omega$ is an irreducible hypersurface.  $($According to Proposition \ref{prop.PV-non-connexe}, this is equivalent to the fact that there exists only one fundamental relative invariant under $G^{\circ}$, up to constants$)$.

{\rm b)}The prehomogeneous vector space $(G,V)$ is called  $2$-irreducible if for any proper invariant subspace $U\subset V$, the prehomogeneous vector space $(G,U)$ has no nontrivial relative invariant.

{\rm c)}The prehomogeneous vector space $(G,V)$ is called quasi-irreducible (abbreviated  Q-irreducible) if for any proper invariant subspace $U\subset V$, the prehomogeneous vector space $(G,U)$ is not regular.

\end{definition}

\begin{rem} It is well known that if $(G,V)$ is irreducible, than there exists at most one fundamental relative invariant. Therefore the irreducible regular $PV$'s with a reductive group are $1$-irreducible. The $PV$ from Example \ref{exemple1} is 1-irreducible but not irreducible.
\end{rem}

\begin{prop}\label{prop-irreductibilite-implications} Let $(G,V)$ be a regular $PV$ where $G$ is reductive. Among the various definitions of irreducibility, we have the following implications:
$$(G,V)  \text{is } 1-\text{irreducible} \Rightarrow (G,V)  \text{ is } 2-\text{irreducible}    \Rightarrow (G,V) \text{ is } Q-\text{irreducible}.$$ 

\end{prop}
\begin{proof} Suppose that $(G,V)$ is not 2-irreducible. Then it exists   a proper invariant subspace $U\subset V$ such that $(G,U) $ has a non trivial relative invariant $f$. Let $W$ be a $G$-invariant supplementary subspace to $U$. Then according to Lemma \ref{lemma-extension-restr-invariants} the fonction $\tilde{f}$ defined by $\tilde{f}(x+y)=f(x)\,\, (x\in U,y\in W)$ is a relative invariant on $V$ depending only on $x$. Therefore the map $\frac{d\tilde{f}}{\tilde{f}}$ cannot be generically surjective. But as $(G,V)$ is regular there exists a relative invariant $\varphi$  such that $\frac{d\varphi}{\varphi}$ is generically surjective. This is not the case if $\varphi= c \tilde{f}^k \, (c\in {\bb C}).$ Therefore there exists another fundamental relative invariant, and hence $(G,V)$ is not $1$-irreducible.  

Suppose now that $(G,V)$ is not  Q-irreducible. Then it exists a proper invariant subspace $U\subset V$ such that $(G,U)$ is regular. Hence $(G,U)$ has a non trivial relative invariant. Therefore $(G,V)$ is not $2$-irreducible.

\end{proof}  
\begin{prop}\label{prop-somme-reg}Let $(G,V)$ be a   $PV$ where $G$ is reductive. Suppose that
$\displaystyle{V=\oplus_{i=1}^{n} V_{i}}$ where each $V_{i}$ is a $G$-invariant subspace such that $(G,V_{i})$ is regular. Let $\Omega$ and $\Omega_{i}$ be the open orbits of $(G,V)$ and $(G,V_{i})$ respectively $(i=1,\dots,n)$. Then $(G,V)$ is regular and $\displaystyle{\Omega= \oplus _{i=1}^{n}\Omega_{i}}$. Moreover any polynomial relative invariant of $(G,V)$ is a product of relative invariants of the spaces $(G,V_{i})$.
\end{prop}

\begin{proof} Let us make the usual identification $\displaystyle{V^*=\oplus_{i=1}^{n}V_{i}^*}$. Let $f_{i}$ be a relatively invariant polynomial  of $(G,V_{i})$ such that $\varphi_{i}=\frac{df_{i}}{f_{i}}: \Omega_{i}\longrightarrow V_{i}^*$ is generically surjective. Replacing eventually $f_{i}$ by its square, we can suppose that $\partial ^{\circ}(f_{i})>1$ ($\partial ^{\circ}(f_{i})$ denotes the degree of $f_{i}$). Define a relative invariant of $(G,V)$ by:
$$f(x_{1},x_{2},\dots,x_{n})=f_{1}(x_{1})f(x_{2})\dots f_{n}(x_{n})\quad (x_{i}\in V_{i}).$$
Then $\displaystyle \varphi(x_{1},x_{2},\dots,x_{n})=\frac{df(x_{1},x_{2},\dots,x_{n})}{f(x_{1},x_{2},\dots,x_{n})}= \varphi_{1}(x_{1})\oplus \varphi_{2}(x_{2})\oplus\dots\oplus \varphi_{n}(x_{n}).$
As the map $x_{i}\longrightarrow \varphi_{i}(x_{i})$ is generically surjective from $\Omega_{i} $ to $V_{i}^*$, we see that $\varphi$ is generically surjective from $\oplus _{i=1}^{n}\Omega_{i}$ to $V^*$. Then from Proposition \ref{prop.PV-regulier-non-connexe} we obtain that $(G,V)$ is regular. Moreover we have $\displaystyle \det d\varphi (x_{1},x_{2},\dots,x_{n})=\prod_{i=1}^{n}\det d\varphi_{i}(x_{i})$ and we know from \cite{Sato-Kimura} p.63 that the Hessian $H_{f}$ is given by
$$H_{f}(x)=(1-r)\det d\varphi(x). f(x)^k$$
where $r= \partial ^{\circ}(f)$ and where $k=\dim V$. Hence $H_{f}\neq 0$ on $\displaystyle \oplus_{i=1}^{n}\Omega_{i}$. On the other hand it is known (\cite{Sato-Kimura}, p.70\footnote{In the paper by M. Sato and T. Kimura, it is written that if $(G,V)$ is a regular $PV $ with $G$ reductive, and if $f$ is a relative invariant with $H_{f}\neq0$, then $\Omega=\{x\,|\,H_{f}(x)\neq0\}$, but analyzing their proof it is easy to see that in fact $ \Omega=\{x\,|\,f(x)H_{f}(x)\neq0\}$ (the first assertion would be wrong if $\partial^{\circ}f=2$)}, \cite{rub-bouquin-PV} p. 22-23$)$,  that if $H_{f}\neq 0$ then $ \Omega=\{x\,|\,f(x)H_{f}(x)\neq0\}$. This implies that $\displaystyle \oplus_{i=1}^{n}\Omega_{i}\subset \Omega$. The reverse inclusion is a consequence of Proposition \ref {prop-components}. The set $S_{i}=V_{i}\setminus \Omega_{i}$ is a hypersurface defined by an equation $P_{i}=0$ where $P_{i} $ is a relatively invariant polynomial on $V_{i}$ (Proposition \ref{prop.PV-regulier-non-connexe}). We will choose $P_{i}$ of minimal degree among the polynomials defining $S_{i}$. Then $P_{i}=f_{i, 1}\dots f_{i,l_{i}}$ where the $f_{i,j}'s$  are irreducible  relatively invariant polynomials   under $G^{\circ }$ on $V_{i}$, which are algebraically independant. From  Proposition \ref{prop.PV-non-connexe} we know that we can write $P_{i}=\varphi_{i,1}\dots \varphi_{i,m_{i}}$, where the $\varphi_{i,j}$'s are polynomials on $V_{i} $ which are relatively invariant under $G$. As  $\displaystyle \Omega= \oplus_{i=1}^{n}\Omega_{i}  $ we obtain that
$$S=V\setminus \Omega=\{ (x_{1},x_{2},\dots,x_{n})\in V\,|\, P(x_{1},x_{2},\dots,x_{n})=\prod_{i=1}^{n}P_{i}(x_{i})=0\}.$$

Using again Proposition \ref{prop.PV-non-connexe}, we obtain that any $G$-relatively invariant polynomial on $V$ is a product of polynomials of the form $\varphi_{i,j}^{\alpha_{i,j}}$, where $\alpha_{i,j}\in {\bb N}$.

\end{proof}

\begin{definition}\label{def-completement-Q-red}  Let $G$ be a reductive group $($not necessarily connected $)$ and let $(G,V)$ be a $PV$. The $PV$ $(G,V)$ is called completely $Q$-reducible if there exists a decomposition $\displaystyle V=\oplus_{i=1}^{n}V_{i}$ where the $V_{i}$'s are $G$-invariant subspaces such that $(G,V_{i})$ is $Q$-irreducible. The spaces $V_{i}$ are then called $Q$-irreducible components of $(G,V)$.
\end{definition}

\begin{rem} We know from Proposition \ref{prop-somme-reg} that a completely $Q$-irreducible $PV$ is regular.
\end{rem}

It is well known that for ordinary finite dimensional completely reducible representations of a group $G$ the equivalence classes occuring in any decomposition into irreducibles are uniquely determined, as well as the isotypic components. Our next aim is to prove analogous results for completely $Q$-reducible regular $PV$'s where the irreducible components are replaced by the $Q$-irreducible components and the isotypic components are replaced by the $Q$-isotypic components.

\begin{theorem}\hfill{}

Let $(G,V)$ be a completely $Q$-reducible $PV$. Let $\displaystyle V=\oplus_{i=1}^{n}V_{i}$ be a decomposition of $V$ into $Q$-irreducible components. Let $W\subset V$ be an invariant subspace such that $(G,W)$ is regular. Then $(G,W)$ is a completely $Q$-reducible $PV$. Moreover if $W_{j}$ is a $Q$-irreducible component of $(G,W)$, there exists an integer $\ell(j) \in \{1,2,\dots,n\}$ such that the representation $(G,W_{j})$ is equivalent to $(G,V_{\ell(j)})$.

The equivalence classes of the $Q$-irreducible components arising in $(G,V)$ are uniquely determined.

Let $\delta$ be an equivalence class of one of  $Q$-irreducible components arising in $\displaystyle V=\oplus_{i=1}^{n}V_{i}$ $($i.e. an equivalence class of one of the representations $(G,V_{j})$$)$. Let $I(\delta)=\{i\,|\, (G,V_{i})\in \delta \}$ and let $m(\delta)={\rm Card}I(\delta)$ be the multiplicity of $\delta$. Let also $\displaystyle V(\delta)= \oplus_{i\in I(\delta)}V_{i}$ be the so-called $Q$-isotypic component of $\delta$. Then $m(\delta)$ does not depend on the decomposition of $V$ into $Q$-irreducible subspaces. Moreover if $U\subset V$ is an invariant subspace of type $\delta$ $($this means that $U$ is a direct sum of $Q$-irreducible invariant subspaces which are all of type $\delta$$)$, then $U$ is a subspace of $V(\delta)$. In other words the $Q$-isotypic components are uniquely determined.
\end{theorem}

\begin{proof} \hfill{}

Let $V_{j}=\oplus_{i=1}^{\ell(j)}U_{j}^{i}$ be a decomposition of $V_{j}$ into irreducible components in the ordinary sense. As we are only interested in equivalence classes of representations we can assume that $W= (\oplus_{j\in A}V_{j})\oplus (\oplus _{j\in A^{c}}\oplus_{i\in I_{j}}U_{j}^{i})$, where $A$ is a subset of $\{1,2,\dots,n\}$ and where $I_{j}$ is a proper subset of $\{1,2,\dots,\ell(j)\}$. After renumbering, we can  suppose that $I_{j}=\{1,2,\dots,m(j)\}$ where $m(j)<\ell(j)$. Let us denote by $x_{j}$ the variable in $V_{j}$ and by $x_{j}^{i}$ the variable in $U_{j}^{i}$. Hence $x_{j}=(x_{j}^1,x_{j}^2,\dots,x_{j}^{\ell(j)})$. Let $j_{1},j_{2},\dots,j_{k}$ be the elements of $A$, and $j_{k+1},\dots,j_{n}$ be the elements in $A^{c}$. 

Let $f$ be a relative invariant   of $(G,W)$ such that $\displaystyle \frac{df}{f}$ is generically surjective. Then  $f$ is a function in the variables:
$$(x_{j_{1}},\dots, x_{j_{k}}; x_{j_{k+1}}^{1},\dots, x_{j_{k+1}}^{m(j_{k+1})};\dots; x_{j_{n}}^{1},\dots,  x_{j_{n}}^{m(j_{n})}).$$
We know from Proposition \ref{prop-somme-reg} that $f$ is a product of relative invariants of the $V_{j}$'s. Hence
$$\displaylines{f(x_{j_{1}},\dots, x_{j_{k}}; x_{j_{k+1}}^{1},\dots, x_{j_{k+1}}^{m(j_{k+1})};\dots; x_{j_{n}}^{1},\dots,  x_{j_{n}}^{m(j_{n})})\cr
=f_{j_{1}}(x_{j_{1}})\dots f_{j_{k}}(x_{j_{k}})f_{j_{k+1}}(x_{j_{k+1}}^1,\dots,x_{j_{k+1}}^{\ell_{j_{k+1}}})\dots f_{j_{n}}(x_{j_{n}}^1,\dots,x_{j_{n}}^{\ell_{j_{n}}})\cr
}$$
where each $f_{j_{r}}$ is a relative invariant of $(G,V_{j_{r}})$.
Therefore:
$$\begin{array}{ccc}
f_{j_{k+1}}& \text{ depends only on the variables }& x_{j_{k+1}}^{1},\dots, x_{j_{k+1}}^{m(j_{k+1})}\\
\vdots& \vdots &\vdots\\
f_{j_{n}}& \text{ depends only on the variables }& x_{j_{n}}^{1},\dots, x_{j_{n}}^{m(j_{n})}.\\
\end{array}$$
But as $\displaystyle \frac{df}{f}=\frac{df_{j_{1}}}{f_{j_{1}}}\oplus \dots \oplus \frac{df_{j_{k}}}{f_{j_{k}}}\oplus \frac{df_{j_{k+1}}}{f_{j_{k+1}}}\oplus \dots \oplus \frac{df_{j_{n}}}{f_{j_{n}}}$ is generically surjective, each $\displaystyle\frac{df_{j_{r}}}{f_{j_{r}}}$ must be generically surjective. For example $\displaystyle \frac{df_{j_{k+1}}}{f_{j_{k+1}}}$ will be generically surjective from an open set of $U_{j_{k+1}}^1\oplus \dots\oplus U_{j_{k+1}}^{m(j_{k+1})}$ to its dual. Therefore, from  Proposition \ref{prop.PV-regulier-non-connexe} we know that $(G, U_{j_{k+1}}^1\oplus \dots\oplus U_{j_{k+1}}^{m(j_{k+1})})$ would be regular. But this is impossible, since $(G,V_{k+1})$ is $Q$-irreducible. Hence $(G,W)\simeq (G,\oplus_{j\in A}V_{j})$, and this shows that $(G,W)$ is completely $Q$-reducible.

Let $W=\oplus_{j=1}^{k}W_{j}$ be a decomposition of $W$ into $Q$-irreducible components. Then the same proof as above, applied to $W_{j}$ instead of $W$ shows that $(G,W_{j})$ is equivalent to $(G,\oplus_{k\in B}V_{k})$, where $B\subset \{1,\dots,n\}$. But as $(G,W_{j})$ is $Q$-irreducible the set $B$ is a single element. The same proof applied to $V$ shows that any $Q$-irreducible component of $V$ is equivalent to some $V_{i}$. Hence the equivalence classes of the $Q$-irreducible components are uniquely determined.

Let us now prove the assertion concerning the multiplicities. Let $V=\oplus_{k=1}^{r}U_{k}$ be another decomposition of $V$ into $Q$-irreducible components. We can suppose that $r\leq n$. From above we know that $(G,U_{1})\simeq (G,V_{i_{1}})$ where $i_{1}\in \{1,\dots,n\}$. Then by a classical argument $(G,\oplus_{k=2}^{n}U_{k})\simeq (G,\oplus_{i\neq i_{1}}V_{i})$. Then inductively one proves that $r=n$ and that there exists a permutation $\sigma$ of $\{1,\dots,n\}$ such that $(G,U_{i})\simeq (G,V_{\sigma(i)})$. Therefore the multiplicity does not depend on the decomposition into $Q$-irreducibles.

Let now $U\subset V$ be an invariant $Q$-irreducible subspace of type $\delta$ and define $V'=\oplus_{i\notin I(\delta)}V_{i}$. Let $S$ be a $G$-invariant supplementary space of $U\cap V(\delta)$ in U. Hence we have:
$$U= U\cap V(\delta)\oplus S,  \quad V=V(\delta)\oplus V'.$$
For $s\in S$  let us write $s=v_{1}+v_{2}$ with $v_{1}\in V(\delta)$ and $v_{2}\in V'$. The linear mapping $\varphi: S \longrightarrow V'$ defined by $\varphi(s)=v_{2}$ is injective, because if $\varphi(s)=v_{2}=0$, then $s=v_{1}\in U\cap V(\delta)\cap S=\{0\}$. Moreover $\varphi$ is $G$-equivariant. Suppose that $U\cap V(\delta)=\{0\}$. If this is the case, we   have $S=U$, and then $S'=\varphi(S)$ is a subspace of type $\delta$ of $V'$. This is not possible from the definition of $V'$. Therefore $U\cap V(\delta)\neq\{0\}$.  Define $U'=U\cap V(\delta) \oplus S'$. As $\varphi$ is $G$-equivariant, the subspace $U'$ is invariant of type $\delta$. Let $f$ be a relative invariant of $(G,U')$ such that $\displaystyle \frac{df}{f}$ is generically surjective. From Proposition \ref{prop-somme-reg} we know that 
$f(x,s')=\varphi_{1}(x)\varphi_{2}(s')$, where $x\in U\cap V(\delta), s'\in S'$, and where $\varphi_{1}$ and $\varphi_{2}$ are relative invariants of $(G,V(\delta))$ and $(G,V')$ respectively. As $\displaystyle \frac{df}{f}=\frac{d\varphi_{1}}{\varphi_{1}}\oplus \frac{d\varphi_{2}}{\varphi_{2}}$, we obtain that $\displaystyle \frac{d\varphi_{1}}{\varphi_{1}}$ and $\displaystyle \frac{d\varphi_{2}}{\varphi_{2}}$ are generically surjective. This implies  that $(G,U\cap V(\delta))$ is regular and this is  possible if and only if $U\cap V(\delta)=U$, because $(G,U)$ is $Q$-irreducible. Hence $U \subset V(\delta)$.

\end{proof}
\section{The decomposition theorem for reductive regular PV's}\hfill 

\subsection{}Of course, reductive regular $PV$'s are not necessarily completely $Q$-reducible as shown by  the following example.
\begin{example}\label{example2}  
 
 Let $n\geq 2$ be an integer and let $G=GL(n,{\bb C})\times {\bb C}^*$ and $V= S(n,{\bb C})\times {\bb C}^{n} $ where $S(n, {\bb C})$ is the space of complex $n$ by $n$ symmetric matrices. The action of $G$ on $V$ is given by
 $$(g,a)(X,v)= (gX {^tg}, a{^tg} ^{-1}v),  \quad g\in GL(n,{\bb C}), a\in {\bb C}^*, X\in S(n,{\bb C}), v\in {\bb C}^n.$$
 The isotropy subgroup of $(I_{n},e_{1})$, where $I_{n}$ is the identity matrix and where $e_{1}$ is the first vector of the canonical basis of ${\bb C}^n$, is easily seen to be isomorphic to the orthogonal group $O(n-1)$. This proves that $(G,V)$ is a reductive regular $PV$. As the irreducible components are $S(n,{\bb C})$ and ${\bb C}^n$, and as $(G,{\bb C}^n)$ is not regular, the $PV$ $(G,V)$ is not completely $Q$-reducible.  
 
   \end{example}
   
   \subsection{Structure of reductive regular PV's}\hfill{}
   
   The following theorem shows the structure of reductive  regular $PV$'s.
   
   \begin{theorem}\label{th-decomposition}\hfill
   
   Let $(G,V)$ be a reductive regular $PV$ and let $x$ be a generic element of $V$. Denote by $G_{x}$ the isotropy subgroup of $x$. There exist a sequence of subspaces $V_{1},V_{2},\dots,V_{n}$ such that $V=V_{1}\oplus V_{2}\oplus\dots\oplus V_{n}$,  a sequence of integers $i_{1}=1<i_{2}<\dots<i_{k}\leq n$ and a sequence of reductive subgroups
   $$G_{x}=G_{k+1}\subset G_{k}\subset\dots\subset G_{1}=G$$
   with the following properties:

{\rm1)} If $x=x_{1}+x_{2}+\dots+x_{n}$ with $x_{j}\in V_{j}$, then
   $$G_{\ell +1}=(G_{\ell})_{x_{i_{\ell}}+\dots+x_{i_{\ell +1}-1}}$$
   
   {\rm 2)} For $\ell \in \{1,\dots,k\}$ the space $V_{i_{\ell}}\oplus\dots\oplus V_{n}$ is $G_{\ell}$-invariant and $(G_{\ell},V_{i_{\ell}}\oplus\dots\oplus V_{n})$ is a regular $PV$.
   
    {\rm 3)} If $i_{\ell}\leq j \leq i_{\ell +1}-1$, then $V_{j}$ is $G_{\ell}$-invariant , $(G_{\ell},V_{j})$ is a $Q$-irreducible $PV$ and  $(G_{\ell}, V_{i_{\ell}}\oplus \dots \oplus V_{i_{\ell +1}-1})$ is a maximal completely $Q$-reducible $PV$ in $V_{i_{\ell}}\oplus \dots \oplus V_{i_{n}}$  . Moreover $V_{i_{\ell +1}}\oplus \dots \oplus V_{n}$ is $G_{\ell}-invariant$ but does not contain any subspace $U\neq\{0\}$ such that $(G_{\ell}, U)$ is regular.
   
   \end{theorem}
   
   \begin{proof} \hfill{}
   
   The proof goes by induction on $\dim V$. There is nothing to prove if $\dim V=1$. Suppose that the theorem is proved for all reductive regular $PV$'s such that $\dim V\leq r$. Let then $(G,V)$ be a reductive regular $PV$ such that $\dim V=r+1$. Let $V'\subset V$ be an invariant subspace such that $(G,V')$ is completely $Q$-reducible and maximal in $V$ for this property. Denote by 
   $$V'= V_{1}\oplus V_{2}\oplus \dots \oplus V_{i_{2}-1}$$
a decomposition of $V'$ into $Q$-irreducible components. Let $V''$ be an invariant supplement  of $V'$. If $V''=\{0\}$ the $PV$ $(G,V)$ is completely $Q$-reducible and the proof is finished. From the maximality of $V'$ and Proposition \ref{prop-somme-reg}, we know that $(G,V'')$ does not contain any subspace $U\neq\{0\}$ such that $(G,U)$ is regular.

Let $x $ be a generic element in $V$. Let us write:
$$x=x_{1}+x_{2}+\dots+x_{i_{2}-1}+x'' \text{ where } x_{j}\in V_{j} \text{ and where } x''\in V''.$$
Define $G_{2}= G_{x_{1}+\dots+x_{i_{2}-1}}$. From Proposition  \ref {prop.PV-regulier-non-connexe} we know that  $G_{2}$ is reductive and from Proposition \ref{prop-components} ${\rm 5)}$ we know that $(G_{2},V'')$ is regular. As $\dim V''\leq r$, we know by induction  that there exists   a sequence of integers $i_{2}<i_{3}<\dots<i_{k}\leq n$  and a sequence of reductive subgroups 
$$G_{x}=(G_{2})_{x''}=G_{k+1}\subset G_{k}\subset \dots \subset G_{2}$$
which have the required properties for the triplet $(G_{2}, V'', x'')$. Then the sequences $i_{1}<i_{2}<i_{3}<\dots<i_{k}\leq n$  and 
$$G_{x}=(G_{2})_{x''}=G_{k+1}\subset G_{k}\subset \dots \subset G_{2}\subset G_{1}=G$$
have the required properties for the triplet $(G,V,x)$.

   \end{proof}
   
   Let us give three examples of the kind of decompositions arising in the preceding Theorem.
   
   \begin{example} Let us return to Example \ref{example2}. In the notations of the preceding Theorem, we take  for $G_{2} $  the isotropy of $I_{n}\in S(n,{\bb C})$, namely $O(n,{\bb C})\times {\bb C}^*$, and $V_{1}=S(n,{\bb C})$ and $V_{2}=V_{i_{2}}={\bb C}^{n}$.
   
   \end{example}
   \begin{example}\label{example-Sato}(Example of the "descending chains" of F. Sato \cite{F-Sato-3})    
   
 Let $V_{m}=M(m+1,m)$ be the space of complex $(m+1)\times m$ matrices.
Define $V=V_{n}\oplus V_{n-1}\oplus \dots \oplus V_{1}$ and let $G=SO(n+1)\times GL(n)\times GL(n-1) \times \dots \times GL(1)$. The group $G$ acts by
$$(g_{n+1}, g_{n},\dots,g_{1})(x_{n},\dots,x_{1})=(g_{n+1}x_{n}g_{n}^{-1}, g_{n}x_{n-1}g_{n-1}^{-1}, \dots, g_{2}x_{1}g_{1}^{-1})$$
where $g_{n+1}\in SO(n+1)$, $g_{i}\in GL(i)$, $x_{i}\in V_{i}$ for $i=1,\dots,n$.
This representation is a regular $PV$ and the fundamental relative invariants are given by 
$$P_{k}(x_{n},x_{n-1},\dots, x_{1})=\det(^t\hskip -2ptx_{k}\hskip 1pt^t \hskip -2ptx_{k-1}\dots^t \hskip -2ptx_{n} x_{n}\dots x_{k-1}x_{k}). $$
This $PV$ is called the $PV$ of descending chains of size $n$ (see \cite{F-Sato-3} for the details). It is then easily seen that $(G,V_{n})$ is a maximal $Q$-completely reducible subspace (in fact it is irreducible regular). Taking $x^0_{n}=\left[\hskip-5pt \begin{array}{c}I_{n}\\0\end{array}\hskip -5pt \right]$ as regular element of $(G,V_{n})$, a simple computation shows that its  isotropy subgroup $G_{x_{n}^0}$ is equal to $D(SO(n)\times SO(n))\times GL(n-1)\times\dots\times GL(1) $ where  $D(SO(n)\times SO(n))$ stands for the diagonal subgroup of $SO(n)\times SO(n)$, the first factor being diagonally embedded in $SO(n+1)$. Therefore the regular $PV$ $(G_{x_{n}^0},V_{n-1}\oplus \dots \oplus V_{1})$ is essentially  the $PV$ of descending chains of size $n-1$. Therefore the sequence of completely $Q$-reducible spaces (under the successive isotropy subgroups) appearing in Theorem \ref{th-decomposition} is $V_{n}, V_{n-1},\dots,V_{1}$.

\end{example}

\begin{example} \label{example-Mortajine} Let $G=GL(2)\times Spin(10)\times {\bb C}^*$ where $Spin(10)$ is the Spin group in dimension $10$. Consider the representation $[\Lambda_{1}\otimes \rho \otimes Id]\oplus [Id\otimes Spin\otimes \Box]$ of $G$ where $\Lambda_{1}$ is the natural $2$-dimensional representation of $GL(2)$, where $\rho$ is the vector representation of $Spin(10)$, where $Spin$ is the half-spin representation of $Spin(10)$, and where $\Box $ is the natural representation  by multiplication of ${\bb C}^*$ on ${\bb C}$. 
 
This representation is a $PV$ whose generic isotropy subgroup is isomorphic to the exceptional simple Lie group $G_{2}$ (see (42) p. 397 of \cite{Kimura-3}). Another argument to prove the prehomogeneity and the regularity is to remark that it corresponds to a $PV$ of parabolic type in $E_{8}$ (see section 4) and that the corresponding grading element is the semi-simple element of an ${\go{sl}}_{2}$-triple (see \cite{Mortajine}, case $E_{8}^3$ in Proposition 6.2.4 a) p.134).
The irreducible subspace $V_{2}$ corresponding to the Spin representation is not regular (Proposition 31 p. 121 in \cite{Sato-Kimura}). The irreducible subspace $V_{1}\simeq {\bb C}^{20}$ of the representation $[\Lambda_{1}\otimes \rho \otimes Id]$ is well known to be regular. Its generic isotropy subgroup is locally isomorphic to $SO(2)\times SO(8)\times {\bb C}^*$ (\cite{Sato-Kimura}, (15) p.145) and the representation $(G_{2}=SO(2)\times SO(8)\times {\bb C}^*, V_{2})$ is regular by Proposition \ref{prop-components}.

\end{example}

 \section{Classification of $Q$-irreducible reductive regular PV's  of parabolic type}
 
 \subsection{PV's of parabolic type} At this point a brief summary of
the theory of $PV$'s of Parabolic type is needed.
 
 The $PV$'s of parabolic type where introduced by the author in \cite{rub-note-PV}, \cite{rub-kyoto}  (see also \cite{rub-these-etat} and \cite{rub-bouquin-PV})

Let ${\go g}$ be a simple complex Lie algebra. Let ${\go  h}$ be a Cartan
subalgebra of ${\go g}$ and denote by $\Sigma$ the set of roots
of $({\go g},{\go  h})$. As usually, for $\alpha\in \Sigma$, we
denote by $H_\alpha$ the corresponding co-root in ${\go  h}$.  We fix once
and for all a system of simple roots
$    \Psi $ for $\Sigma$. We denote by $\Sigma^+$ 
(resp. $\Sigma^-$) the corresponding set of positive (resp. negative)
roots in $\Sigma  $. Let
$\theta$ be a
subset of
$\Psi$ and let us   make the standard construction of the   parabolic subalgebra ${\go 
p}_\theta
\subset
{\go g}$ associated to $\theta$. As usual we   denote by $\langle
\theta\rangle$ the set of all roots which are linear combinations of elements in
$\theta$, and put $\langle \theta\rangle^\pm=\langle \theta\rangle\cap \Sigma^\pm$.

Set
\begin{align*} {\go  h}_\theta=\theta^\bot=\{X\in {\go  h}\,|\,\alpha(X)=0
\;\; \forall \alpha\in \theta\},\qquad   &{\go  h}(\theta) =\sum_{\alpha\in \theta} {\bb 
C}H_\alpha\\
  {\go  l_\theta}={\go  z}_{{\go g}}({\go  h}_\theta)= {\go  h}\oplus
\sum_{\alpha\in \langle \theta\rangle} {\go g}^\alpha,\qquad  \qquad  \qquad \qquad  &{\go 
n}_\theta^\pm=\sum_{\alpha\in \Sigma^\pm\setminus
\langle\theta\rangle^\pm} {\go g}^\alpha
\end{align*}

Then ${\go  p}_\theta= {\go  l}_\theta \oplus {\go  n}_\theta^+$ is called the standard
parabolic subalgebra  associated to $\theta$. There is also a standard ${\bb Z}$-grading
of ${\go g}$ related to these data. Define $H_\theta$ to be the unique element
of ${\go  h}_\theta$ satisfying the linear equations
$$\alpha(H_\theta)=0 \quad \forall \alpha\in \theta\quad {\rm and}  $$
$$\alpha(H_\theta)=2 \quad \forall \alpha\in \Psi\setminus \theta.\quad    $$
The before mentioned grading is just the grading obtained from the eigenspace
decomposition of $\ad H_\theta$:
$$d_p(\theta)=\{X\in {\go g} \,|\,[H_\theta,X]=2pX\}.$$
Then we obtain easily:
$${\go g}=\oplus_{p\in{\bb Z}}d_p(\theta),\quad{\go 
l}_\theta=d_0(\theta),\quad{\go  n}_\theta^+=\sum_{p\geq 1} d_p(\theta),\quad{\go 
n}_\theta^-=\sum_{p\leq -1} d_p(\theta).$$
It  is known that $({\go  l}_\theta,d_1(\theta))$
is a prehomogeneous vector space (in fact all the spaces $({\go 
l}_\theta,\,d_p(\theta))$ with ${p\not =0}$ are prehomogeneous, but there is no loss of
generality if we only consider $({\go  l}_\theta,\,d_1(\theta)))$. These  spaces have
been called prehomogeneous vector spaces of parabolic type (\cite{rub-note-PV}). There are in general
neither irreducible nor regular. But they are of particular interest, because in the
parabolic context, the group (or more precisely its Lie algebra ${\go  l}_\theta$) and
the space (here $d_1(\theta))$ of the $PV$ are embedded into  a rich structure, namely the
simple Lie algebra
${\go g}$. For example the derived representation of the $PV$ is just the   adjoint
representation  of
${\go  l}_\theta$ on $d_1(\theta)$. Moreover the Lie algebra ${\go g}$ also
contains the dual $PV$, namely $({\go  l}_\theta,d_{-1}(\theta))$.

It may be worthwhile  noticing also that 
$d_1(\theta)=\sum_{\beta\in\sigma_1}{\go g}^\beta$, where $\sigma_1$ is the set
of roots which belong to the set $(\Psi \setminus \theta)\,+\,{\bb Z} \theta $, where ${\bb Z}\theta$ is the ${\bb Z}$--span
of $\theta$.

As these $PV$'s are in one to one correspondence with the subsets $\theta \subset \Psi$, we
make the convention to describe them by the mean of the following weighted Dynkin diagram:

\begin{definition}\label{def-diagramme}  
 The diagram of the $PV$ $({\go  l}_\theta,d_1(\theta))$
is the Dynkin diagram of $({\go g},{\go  h})$ $($or $\Sigma$ $)$, where the
vertices corresponding to the simple roots of $\Psi\setminus \theta$ are circled $($see an example
below$)$.

\end{definition}

 This very simple classification  by means of diagrams contains nevertheless some
immediate and  interesting informations concerning the $PV$ $({\go  l}_\theta,d_1(\theta))$ (for all these
facts, see \cite{rub-note-PV}, \cite{rub-kyoto} or \cite{rub-these-etat}):

 \noindent $\bullet$ The Dynkin diagram of $ {\go  l}_\theta'=[{\go 
l}_\theta,{\go  l}_\theta]   $ (\ie the semi-simple part of the Lie algebra of the group)
is the Dynkin diagram of ${\go g}$ where we have removed the circled vertices
and the edges connected to these vertices.

\noindent $\bullet$ In fact as a Lie algebra ${\go  l_\theta ={\go 
l}_\theta}'\oplus{\go  h}_\theta$ and $\dim {\go  h}_\theta=$ the number of circled
vertices.

\noindent $\bullet$ The number of irreducible components of the representation $({\go l
}_\theta,\,d_1(\theta))$ is also equal to the number of circled roots. More precisely, if
$\alpha$ is a (simple) circled root, then any nonzero root vector $X_\alpha\in 
{\go g}^\alpha$ generates an irreducible ${\go  l}_\theta$--module $V_\alpha$,
and $d_1(\theta)=\oplus_{\alpha\in\Psi\setminus \theta} V_\alpha$ is the decomposition of
$d_1(\theta)$ into irreducibles.

In fact  the decomposition of the representation (${\go  l}_\theta ,d_1(\theta)$) into
irreducibles can also be described by using the eigenspace  decomposition with respect to
$\ad({\go  h}_\theta)$. Let me explain this. For each $\alpha\in {\go  h}^*$, let
$\overline{\alpha}$ be the restriction of $\alpha$ to ${\go  h}_\theta$ and define 
$${\go g}^{\overline{\alpha}}=\{X\in {\go g}\,|\, \forall H \in {\go 
h}_\theta ,\, [H,X]= \overline{\alpha}(H)X\}.$$
Then ${\go g}^{\overline 0}={\go  l}_\theta$ and for $\alpha \in \Psi \setminus
\theta$, we have $V_\alpha = {\go g}^{\overline{\alpha}}$. Hence we can write
$$d_1(\theta)=\oplus_{\alpha\in \Psi\setminus \theta}{\go g}^{\overline{\alpha}}.$$

 Moreover one can notice (always for $\alpha \in \Psi\setminus \theta$) that
$V_\alpha={\go g}^{\overline{
\alpha}}=\sum_{\beta\in\sigma_1^\alpha}{\go g}^\beta$, where
$\sigma_1^\alpha$ is the set of roots which belong to $\alpha\,+\,\langle \theta\rangle$.

\noindent $\bullet$ Moreover one can directly read the highest weight of $V_\alpha$ from
the diagram. The highest weight of $V_\alpha$ relatively to the Borel sub-algebra ${\go 
b}_\theta^-={\go  h}\oplus\sum_{\alpha\in\langle \theta\rangle^-}{\go g}^\alpha$ is $\overline{\alpha}=\alpha_{|_{{\go  h}(\theta)}}$. Let $\omega_\beta$
$(\beta\in \theta)$ be the fundamental weights of ${\go  l}_\theta'$ (\ie the dual basis
of $(H_\beta)_{\beta\in \theta}$). For each circled  root $\alpha$ (\ie for
each
$\alpha\in \Psi\setminus \theta $\,), let $J_\alpha=\{(\beta_i)\} $ be the set of roots
in $\theta$ (= non-circled) which are connected to $\alpha$ in the diagram. From
elementary diagram considerations we know that $J_\alpha$ may be empty and that there are
always less than
$3$ roots in
$J_\alpha$.

          If $J_\alpha=\emptyset$, then $V_\alpha$ is the trivial one dimensional
representation of ${\go  l}_\theta$. 
  
If $J_\alpha\not =\emptyset$, then
$\overline{\alpha}=\sum_{i\in J_\alpha} c_i\omega_{\beta_i}$ where
$c_i=\alpha(H{_{\beta_i}})$ and $  \alpha( H_{\beta_i})$ can be computed as follows:

$$ (R)
\begin{cases}
\text{ if  } ||\alpha||\leq||\beta_i||  , \text{ then }  \alpha(H_{\beta_i})=-1\ ;\\

\text{ if  }   ||\alpha||>||\beta_i||\text{  and if } \alpha \text{ and  } \beta_i \text{ are
connected by }
 j \text{ arrows }  ( 1\leq j\leq 3), \\
  \quad\text{ then }  \alpha(H_{\beta_i})=-j\ . 
\end{cases}
$$

\medskip\medskip

Let us illustrate this with an example.

\begin{example} \label{exdiagram}  Consider the following diagram:
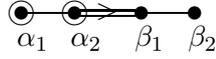
\begin{align*}\hbox{\unitlength=0.5pt
\begin{picture}(200,30)
\put(10,10){\circle*{10}}
\put(10,10){\circle{18}}
\put(7,-15){$\alpha_1$}
\put(15,10){\line (1,0){30}}
\put(50,10){\circle*{10}}
\put(50,10){\circle{18}}
\put(47,-15){$\alpha_2$}
\put(54,12){\line (1,0){41}}
\put(54,8){\line (1,0){41}}
\put(66,5.5){$>$}
\put(100,10){\circle* {10}}
\put(95,-15){$\beta_1$}
\put(105,10){\line (1,0){30}}
\put(140,10){\circle*{10}}
\put(135,-15){$\beta_2$}
\end{picture}} 
\end{align*}
 
\medskip
The preceding diagram is the diagram of a $PV$ of parabolic type inside ${\go g}\simeq F_4$.
 Here we have $\theta=\{\beta_1,\beta_2\}$ and $\Psi\setminus
\theta=\{\alpha_1,\alpha_2\}$. The Lie algebra
${\go  l}_\theta$ is isomorphic to
$A_2\oplus{\go  h}_\theta$ where $\dim {\go  h}_\theta$ = number of circled roots = $2$.
As $J_{\alpha_1}=\emptyset$, the representation of ${\go l }_\theta'$ on $V_{\alpha_1}$
is the trivial representation. Hence the action of $ {\go  l}_\theta$ on $V_{\alpha_1}$
reduces to the character of ${\go  h}_\theta$ given by the restriction of the root
$\alpha_1$ to ${\go  h}_\theta$. On the other hand we have $J_{\alpha_2}=\{\beta_1\}$.
Therefore, applying the rules $(R)$ above, we see that  $V_{\alpha_2}$ is the irreducible $ A_2 $--module with highest
weight $-2\omega_1$, where $\{\omega_1,\omega_2\}$ is the set of fundamental weights of
$A_2$ corresponding to $\beta_1$ and $\beta_2$. Again the action of ${\go  h}_\theta$ on
$V_{\alpha_2}$ is scalar with eigenvalue the restriction of $\alpha_2$ to ${\go 
h}_\theta$.
\end{example}

\medskip\medskip\medskip

One can prove (\cite {rub-note-PV}) that the $PV$ of parabolic type $({\go  l}_\theta,d_1(\theta))$
is irreducible if and only if ${\go  p}_\theta$ is a maximal parabolic subalgebra, {\it
i.e.} if and only if $\Psi \setminus \theta$ is reduced to  a single root $\alpha_1$.

  The $PV$'s of parabolic type which are
irreducible and regular were classified by the list of the    "weighted" Dynkin
diagram of
${\go g}$, where the root $\alpha_1$ in the discussion above is circled.
This classification was announced  first in \cite{rub-note-PV}  and then given explitly  in \cite{rub-kyoto} and \cite{rub-these-etat} (see also the book \cite{rub-bouquin-PV}).

\begin{rem}
Of course the irreducible regular $PV$'s of parabolic type are $Q$-irreducible. Therefore in order to complete the classification of the $Q$-irreducible $PV$'s of parabolic type, it is enough to consider only $PV$'s which are reducible. This will be done in the sequel of the paper.
\end{rem} 

For further use we need also to introduce the notion of subdiagram of the weighted Dynkin diagram associated to $(\Psi,\theta)$. Let $\Gamma$ be a subset of $\Psi\setminus \theta$, that is a subset of the circled roots. For $\alpha \in \Gamma$ define $\Psi_{\alpha}$ to be the connected component of $\theta\cup \{\alpha\}$ containing $\alpha$. Define then
$$\Psi_{\Gamma}=\cup_{\alpha\in \Gamma}\Psi_{\alpha}\text{ and }\theta_{\Gamma}=\theta\cap \Psi_{\Gamma}.$$

\begin{definition}\label{def-sous-diag} \hfill{}

 The weighted Dynkin diagram associated to the pair $(\Psi_{\Gamma},\theta_{\Gamma})$ is called a subdiagram of the diagram associated to $(\Psi,\theta)$.
 \end{definition}
 
 It can be noticed that a subdiagram is just a subset $\Gamma$ of the circled roots togeteher with the non-circled roots which are connected to a root in $\Gamma$ (through a path in the non-circled roots). It may also be noticed that the subdiagrams of a connected diagram are not necessarily connected. Let us give an example.
 
 \begin{example}\label{examplesubdiagrams}: Consider the following weighted diagram in
$D_9$

\medskip\medskip 
  \hbox{\unitlength=0.5pt
\begin{picture} (0,0) (-150,20)
\put(-50,5){$D=$}\put(10,10){\circle*{10}}
\put(5,-15){$\beta_1$}

\put(15,10){\line (1,0){30}}
\put(50,10){\circle*{10}}
\put(45,-15){$\alpha_1$}
\put(50,10){\circle{18}}
\put(55,10){\line (1,0){30}}
\put(90,10){\circle*{10}}
\put(85,-15){$\alpha_2$}
\put(90,10){\circle {18}}
\put(95,10){\line (1,0){30}}
\put(130,10){\circle*{10}}
\put(125,-15){$\beta_2$}
\put(135,10){\line (1,0){30}}
\put(170,10){\circle*{10}}
 
\put(165,-15){$\alpha_3$}
\put(170,10){\circle {18}}
\put(175,10){\line (1,0){30}}
\put(210,10){\circle*{10}}
\put(205,-15){$\beta_3$}
\put(215,10){\line (1,0){30}}
\put(240,10){\circle*{10}}
 
\put(240,14){\line (1,1){20}}
\put(265,36){\circle*{10}}
\put(280,36){$\alpha_4$}
\put(265,36){\circle {18}}
\put(250,7){$\beta_4$}
\put(240, 6){\line(1,-1){20}}
\put(265,-16){\circle*{10}}
\put(275,-20){$\beta_5$}
\end{picture}
} 
\medskip\medskip\medskip\medskip\medskip
  
where $\theta=\{\beta_1,\beta_2,\beta_3,\beta_4,\beta_5\}$ and $\Psi\setminus
\theta=\{\alpha_1,\alpha_2,\alpha_3,\alpha_4\}$.

We have:
\hbox{\unitlength=0.5pt
\begin{picture}(0,0)(-90,20)
\put(-50,-10){$ \theta\cup\{\alpha_1\}=$ 
 \put(10,5){\circle*{10}}
\put(5,-20){$\beta_1$}

\put(15,5){\line (1,0){30}}
\put(50,5){\circle*{10}}
\put(45,-20){$\alpha_1$}
\put(50,5){\circle{18}}}
 
\put(145,-5){\circle*{10}}
\put(140,-35){$\beta_2$}
\put(180,-5){\circle*{10}}
\put(175,-35){$\beta_3$}
\put(185,-5){\line (1,0){30}}
 \put(220,-5){\circle*{10}}
\put(215,-35){$\beta_4$}
\put(225,-5){\line (1,0){30}}
\put(260,-5){\circle*{10}}
\put(255,-35){$\beta_5$}
\end{picture}}

 \vskip 40pt
Therefore the irreducible subdiagram associated to $\{\alpha_1\}$ is given by:
\hbox{\unitlength=0.5pt
\begin{picture}(0,0)(-90,20)
\put(-500,-36){$D_{\{\alpha_1\}}=$}
 \put(-395,-30){\circle*{10}}
\put(-400,-60){$\beta_1$}
\put(-390,-30){\line (1,0){30}}
\put(-355,-30){\circle*{10}}
\put(-360,-60){$\alpha_1$}
\put(-355,-30){\circle{18}} 
\end{picture} }
\vskip 50pt

Similarly the   subdiagrams of $D$ corresponding to $\Gamma=\{\alpha_{1}, \alpha_{4}\}$ and $\Gamma=\{\alpha_{3},\alpha_{4}\}$ are respectively: 

  \hbox{\unitlength=0.5pt
\begin{picture} (0,0) (-150,20)
\put(-40,5){$D_{\{\alpha_{1},\alpha_{4}\}}=$}
\put(75,10){\circle*{10}}
\put(65,-15){$\beta_1$}

\put(80,10){\line (1,0){30}}
\put(120,10){\circle*{10}}
\put(105,-15){$\alpha_1$}
\put(120,10){\circle{18}}
 
\put(205,10){\circle*{10}}
\put(195,-15){$\beta_3$}
\put(210,10){\line (1,0){30}}
\put(240,10){\circle*{10}}
 
\put(240,14){\line (1,1){20}}
\put(265,36){\circle*{10}}
\put(280,36){$\alpha_4$}
\put(265,36){\circle {18}}
\put(250,7){$\beta_4$}
\put(240, 6){\line(1,-1){20}}
\put(265,-16){\circle*{10}}
\put(275,-20){$\beta_5$}
\end{picture}
} 
\vskip 30pt

  \hbox{\unitlength=0.5pt
\begin{picture} (0,0) (-150,20)
\put(15,5){$D_{\{\alpha_{3},\alpha_{4}\}}=$}
  
\put(130,10){\circle*{10}}
\put(125,-15){$\beta_2$}
\put(135,10){\line (1,0){30}}
\put(170,10){\circle*{10}}
 
\put(165,-15){$\alpha_3$}
\put(170,10){\circle {18}}
\put(175,10){\line (1,0){30}}
\put(210,10){\circle*{10}}
\put(205,-15){$\beta_3$}
\put(215,10){\line (1,0){30}}
\put(240,10){\circle*{10}}
 
\put(240,14){\line (1,1){20}}
\put(265,36){\circle*{10}}
\put(280,36){$\alpha_4$}
\put(265,36){\circle {18}}
\put(250,7){$\beta_4$}
\put(240, 6){\line(1,-1){20}}
\put(265,-16){\circle*{10}}
\put(275,-20){$\beta_5$}
\end{picture}
}

\vskip 30pt

 \end{example}

\begin{definition} A weighted Dynkin diagram will be called {\it regular} (resp. {\it Q-irreducible}) if the corresponding $PV$ of parabolic type is regular (resp. $Q$-irreducible).
\end{definition}

 
 \subsection{Classification of classical $Q$-irreducible reductive regular PV's of parabolic type}\hfill{}
 
 We adopt the following numbering of the roots for classical simple Lie algebras

\vskip 20pt

\hbox{\unitlength=0.5pt
\begin{picture}(300,30)
\put(90,10){\circle*{10}}
\put(85,-10){$\alpha_1$}
\put(95,10){\line (1,0){30}}
\put(130,10){\circle*{10}}
\put(140,10){\circle*{1}}
\put(145,10){\circle*{1}}
\put(150,10){\circle*{1}}
\put(155,10){\circle*{1}}
\put(160,10){\circle*{1}}
\put(165,10){\circle*{1}}
\put(170,10){\circle*{1}}
\put(175,10){\circle*{1}}
\put(180,10){\circle*{1}}
 \put(195,10){\circle*{10}}
 \put(195,10){\line (1,0){30}}
\put(230,10){\circle*{10}}

\put(235,10){\line (1,0){30}}
\put(270,10){\circle*{10}}
\put(260,-10){$\alpha_n$}

\put(430,10){$A_{n}$}
\end{picture}
}

\medskip\medskip

\medskip
\hbox{\unitlength=0.5pt
\begin{picture}(300,30)(-82,0)
\put(10,10){\circle*{10}}
\put(0,-10){$\alpha_1$}
 
\put(15,10){\line (1,0){30}}
\put(50,10){\circle*{10}}
\put(55,10){\line (1,0){30}}
\put(90,10){\circle*{10}}
\put(95,10){\line (1,0){30}}
\put(130,10){\circle*{10}}
\put(135,10){\circle*{1}}
\put(140,10){\circle*{1}}
\put(145,10){\circle*{1}}
\put(150,10){\circle*{1}}
\put(155,10){\circle*{1}}
\put(160,10){\circle*{1}}
\put(165,10){\circle*{1}}
\put(170,10){\circle*{10}}
\put(174,12){\line (1,0){41}}
\put(174,8){\line (1,0){41}}
\put(190,5.5){$>$}
\put(220,10){\circle*{10}}
\put(210,-10){$\alpha_n$}
\put(355,10){$B_n$ }
\end{picture}
}

\medskip\medskip

 \hbox{\unitlength=0.5pt
\begin{picture}(300,30)(-82,0)
\put(10,10){\circle*{10}}
\put(0,-10){$\alpha_1$}

\put(15,10){\line (1,0){30}}
\put(50,10){\circle*{10}}
\put(55,10){\line (1,0){30}}
\put(90,10){\circle*{10}}
\put(95,10){\circle*{1}}
\put(100,10){\circle*{1}}
\put(105,10){\circle*{1}}
\put(110,10){\circle*{1}}
\put(115,10){\circle*{1}}
\put(120,10){\circle*{1}}
\put(125,10){\circle*{1}}
\put(130,10){\circle*{1}}
\put(135,10){\circle*{1}}
\put(140,10){\circle*{10}}
\put(145,10){\line (1,0){30}}
\put(180,10){\circle*{10}}
\put(184,12){\line (1,0){41}}
\put(184,8){\line(1,0){41}}
\put(190,5.5){$<$}
\put(220,10){\circle*{10}}
\put(210,-10){$\alpha_n$}
 
\put(355,10){$C_{n}$} 
\end{picture}
}

\medskip\medskip\medskip

\hbox{\unitlength=0.5pt
\begin{picture}(300,40)(-82,-10)
\put(10,10){\circle*{10}}
 \put(0,-10){$\alpha_1$}
\put(15,10){\line (1,0){30}}
\put(50,10){\circle*{10}}
\put(55,10){\line (1,0){30}}
\put(90,10){\circle*{10}}
\put(95,10){\line (1,0){30}}
\put(130,10){\circle*{10}}
\put(140,10){\circle*{1}}
\put(145,10){\circle*{1}}
\put(150,10){\circle*{1}}
\put(155,10){\circle*{1}}
\put(160,10){\circle*{1}}
\put(170,10){\circle*{10}}
 
\put(175,10){\line (1,0){30}}
\put(210,10){\circle*{10}}
\put(217,5){$\alpha_{n-2}$}
\put(215,14){\line (1,1){20}}
\put(240,36){\circle*{10}}
\put(247,34){$\alpha_{n-1}$}
\put(215,6){\line (1,-1){20}}
\put(240,-16){\circle*{10}}
\put(247,-20){$\alpha_{n}$}
\put(355,10){$D_n$ }
\end{picture}
}
\vskip 20pt 

The classification of $Q$-irreducible $PV$'s in the classical simple Lie algebras need now some technical lemmas. If $\omega_{i}$ is the fundamental weight corresponding to the root $\alpha_{i}$, we denote by $\Lambda_{i}({\go g})$ the corresponding representation of ${\go g}$. If this representation can be lifted to a group $G$ with Lie algebra ${\go g}$, we will denote by $\Lambda_{i}(G)$ the lifted representation of $G$. For example   we will denote by $\Lambda_{1}(GL(n))$ (resp. by $\Lambda_{n}(GL(n))$) the natural representation of $GL(n)$ on ${\bb C}^n$ (resp. the dual of the natural representation of $GL(n)$ on ${\bb C}^n$).

\begin{lemma}\label{lemme-GXGL} Let $G$ be a simple classical group. Let $d_{1}=\dim \Lambda_{1}(G)$. Let $n\leq d_{1}$ and consider the $PV$ $(G\times GL(n),\Lambda_{1}(G)\otimes \Lambda_{n}(GL(n)))$ $($it is a $PV$ because it is parabolic$)$. Then either this $PV$ is regular, or there exists a normal unipotent subgroup of the generic isotropy subgroup which is included in $G$.
\end{lemma}

\begin{proof}

If $G$ is of type $A_{k}$ then an obvious calculation shows the Lemma.

If $G$ is of type $B_{k}$ or $D_{k}$, then we know from table $1$ in \cite{rub-kyoto}, that the given $PV$ is always regular.

The same argument holds if $G$ is of type $C_{k}$ and if $n$ is even.

If $G$ is of type $C_{k}$ and if $n$ is odd, the space is not regular and the calculations made at p. $102$ of \cite{Sato-Kimura} show the assertion concerning the normal unipotent subgroup.

\end{proof}

\begin{lemma}\label{lemme-GLXGLXG}Let $G$ be a reductive algebraic group and let $\Lambda$ be an  representation of $G$ of dimension $r$. Let $p$ and $q$ be two integers such that $p<q$ and $r<q$. Suppose that the representation $[\Lambda_{p-1}(GL(p))\otimes   \Lambda_{1}(GL(q))]\oplus [\Lambda_{q-1}(GL(q))\otimes \Lambda]$ of the group $GL(p)\times GL(q)\times G$ is prehomogeneous $($this is automatically the case if $p\geq r$$)$. Then:

$1)$ If $p\neq r$, the preceding $PV$ is not regular and there exists a non-trivial normal unipotent subgroup of the generic isotropy subgroup which is included in $GL(q)$.

$2)$ If $p=r$, the preceding $PV$ is regular and $1$-irreducible $($hence $Q$-irreducible from PropositionÊ $\ref{prop-irreductibilite-implications})$.
\end{lemma}  

\begin{proof} As $G$ only acts through its representation $\Lambda(G)$, we can assume that $G\subset GL(r)$. The space of the representation is  $M(q,p)\oplus M(r,q)$ (where $M(u,v)$ stands for the space of $u\times v$ matrices), and the group $GL(p)\times GL(q)\times G$ acts by
$$(g_{1},g_{2},g_{3})(X,Y)=(g_{2}Xg_{1}^{-1}, g_{3}Yg_{2}^{-1})$$
where $g_{1}\in GL(p), g_{2}\in GL(q), g_{3}\in G, X\in M(q,p), Y\in M(r,q).$

As usually we denote by $\Omega$ the open orbit in $M(q,p)\oplus M(r,q)$.
\vskip 5pt 
$ \bullet$  {\it Suppose first that $p<r$}.

As the representation is supposed to be prehomogeneous, we know from Proposition \ref{prop-components} that the open orbits of the components are the matrices of maximal rank in $M(q,p)$ and $M(r,q)$ respectively. Let $X_{0}=\left[\hskip -5pt\begin{array}{c}
I_{p}\\
0\end{array}\hskip -5pt\right] \in M(q,p)$ where $I_{p}$ is the identity matrix of size $p$. An easy calculation shows that the isotropy subgroup of $(X_{0},0)\in M(q,p)\oplus M(r,q)$ is the set of matrices of the form:
$$(g_{1}, \left [\hskip -5pt\begin{array}{cc}g_{1}&B\\
0&D\end{array}\hskip -5pt \right], g_{3}),\,\,  g_{1}\in GL(p), D\in GL(q-p), B\in M(p,q-p), g_{3}\in G.$$
It can also be easily seen that that the set $\cal O$ of matrices of the form
$$\left[\hskip -5pt\begin{array}{c} u\,|\,0
 \end{array}\hskip -5pt \right] . \left [\hskip -5pt\begin{array}{cc}g_{1}&B\\
0&D\end{array}\hskip -5pt\right]$$
where $ g_{1}\in GL(p), D\in GL(q-p), B\in M(p,q-p), u\in GL(r), \left[\hskip -5pt\begin{array}{c} u\,|\,0
 \end{array}\hskip -5pt \right]\in M(r,q)$ contains a  Zariski open subset  of  $M(r,q)$.
 
 Therefore  ${\cal O}\cap \{m \in M(r,q)\,|\, (X_{0},m)\in \Omega\}\neq \emptyset$

This implies  that there exists a generic element of the form $(X_{0},Y_{0})$ where $Y_{0}=(y_{0},0)$ with $y_{0}\in GL(r)$.
Again a simple calculation shows that the isotropy subgroup of $(X_{0},Y_{0})$ is the set of triplets  of the form:
$$(g_{1}, \begin{array}{c}\left[\hskip -3pt\begin{tabular}{l|c}
$\begin{array}{cc}g_{1}&B_{1}\\
0&D_{1}\end{array}$&0\\
\hline
\begin{tabular}{c c}
\hskip 3pt 0&$D_{2}$
\end{tabular}&$D_{3}$
\end{tabular} \hskip -3pt\right ]
\end{array}, g_{3})$$
where $g_{1}\in GL(p), D_{1}\in GL(r-p), D_{2}\in M(q-r,r-p), D_{3}\in GL(q-r), g_{3}\in G\subset GL(r)$ and where 
$$y_{0}.\left [\hskip -5pt\begin{array}{cc}g_{1}&B_{1}\\
0&D_{1}\end{array}\hskip -5pt\right ]=g_{3}.y_{0}.$$
It is now clear that the set of triplets  of the form
$$(I_{p}, \begin{array}{c}\left[\hskip -3pt\begin{tabular}{c|c}
 $I_{r}$&0\\
\hline
\begin{tabular}{c c}
\hskip 3pt 0&$D_{2}$
\end{tabular}&$I_{q-r}$
\end{tabular}\hskip -3pt\right ]
\end{array}, I_{r})$$
is a unipotent normal subgroup of the (generic)  isotropy subgroup of $(X_{0},Y_{0})$.

 \vskip 5pt
 
 $ \bullet$  {\it Suppose  that $p> r$}.
 
 Let $X_{0}=\left[\hskip -5pt\begin{array}{c}
I_{p}\\
0\end{array}\hskip -5pt\right] \in M(q,p)$ and let  $Y_{0}=\left[\hskip -5pt \begin{array}{rl}
I_{r}&\hskip -5pt 0 \end{array}\hskip -5pt\right] \in M(r,q)$. The isotropy subgroup of $(X_{0},Y_{0})$ is the set of triples of matrices of the form 

$$(\left[\hskip -5pt \begin{array}{cc}g_{3}&0\\
C_{1}&D_{1}\end{array}\hskip -5pt\right], \hskip -5pt\begin{array}{c}\left[\hskip -3pt\begin{tabular}{l|c}
 $\begin{array}{cc}g_{3}&0\\
C_{1}&D_{1}\end{array}$&\begin{tabular}{c}0\\
 
$D_{2}$\end{tabular}\\
\hline
\begin{tabular}{c c}
\hskip 3pt 0&\hskip 7pt $0$
\end{tabular}&$D_{3}$
\end{tabular}\hskip -3pt\right ]  \end{array}\hskip -3pt, g_{3}),$$

where $g_{3}\in G\subset GL(r), D_{1}\in GL(p-r), C_{1}\in M(p-r,r), D_{3}\in GL(q-p), D_{2}\in M(p-r, q-p)$.

A simple calculation of dimensions shows now that  the representation is prehomogeneous and that $(X_{0},Y_{0})$ is generic. Of course the set of triplets of the form 

$$(I_{p},\left[\hskip -3pt\begin{tabular}{r|l}$I_{p}$&\begin{tabular}{c}0\\
$D_{2}$\end{tabular}\\
\hline
0&$I_{q-p}$\end{tabular}\hskip -3pt\right],I_{r})$$
is a unipotent normal subgroup of the (generic)  isotropy subgroup of $(X_{0},Y_{0})$.

 \vskip 5pt
 
 $ \bullet$  {\it  Finally suppose  that $p= r$}.
 
  Let $X_{0}=\left[\begin{array}{c}
I_{p}\\
0\end{array}\right] \in M(q,p)$ and let  $Y_{0}=\left[\hskip -5pt \begin{array}{rl}
I_{p}&\hskip -5pt 0 \end{array}\hskip -5pt\right] \in M(p,q)$. The isotropy subgroup of $(X_{0},Y_{0})$ is the set of triplets of the form 
$$(g_{3}, \left[\hskip -5pt\begin{array}{cc}g_{3}&0\\
0&D
\end{array}\hskip -5pt\right],g_{3}) \eqno (*)$$ 

where $ g_{3}\in G \subset GL(p), \, D\in GL(q-p)$.

Again an easy computation of dimensions shows that this representation is prehomogeneous. As the generic isotropy subgroup is reductive,  this $PV$ is regular. 

Let $G_{1}$ be the subgroup of $GL(p)\times GL(q)\times G$ generated by a generic isotropy subgroup and by the commutator subgroup $SL(p)\times SL(q)\times G'$. The characters of the relative invariants are exactly those characters which are trivial on $G_{1}$ (this is true for any PV). From $(*)$ it is easy to see that $G/G_{1}$ is always a one dimensional torus, hence there exists only one fundamental relative invariant. One can remark that this invariant is given by $f(X,Y)=\det(YX), \, X\in M(p,q), Y\in M(q,p)$.

 \end{proof}
 
 \begin{lemma}\label{lemme-GLXGL-in-D(1)} Consider the representation 
 $$[\Lambda_{p-1}(GL(p))\otimes \Lambda_{r-1}(GL(r))]\oplus [Id(GL(p)\otimes \Lambda_{2}(GL(r))]$$
 of the group $GL(p)\times GL(r)$, with $r\geq 3$. Note that this representation is prehomogeneous since it is infinitesimally equivalent to the $PV$ of parabolic type associated to the diagram
 $$ {
\hbox{\unitlength=0.5pt
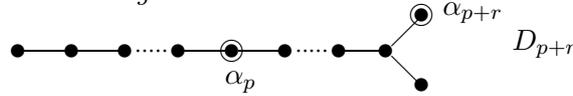
\begin{picture}(240,35)(0,-15)
\put(10,0){\circle*{10}}
\put(15,0){\line (1,0){30}}
\put(50,0){\circle*{10}}
\put(55,0){\line (1,0){30}}
\put(90,0){\circle*{10}}
\put(100,0){\circle*{1}}
\put(105,0){\circle*{1}}
\put(110,0){\circle*{1}}
\put(115,0){\circle*{1}}
\put(120,0){\circle*{1}}
\put(130,0){\circle*{10}}
\put(135,0){\line (1,0){30}}
\put(170,0){\circle*{10}}
\put(170,0){\circle{16}}
\put(165,-25){$\alpha_p$}
\put(175,0){\line (1,0){30}}
\put(210,0){\circle*{10}}
\put(220,0){\circle*{1}}
\put(225,0){\circle*{1}}
\put(230,0){\circle*{1}}
\put(235,0){\circle*{1}}
\put(240,0){\circle*{1}}
\put(250,0){\circle*{10}}
\put(255,0){\line (1,0){30}}
\put(285,0){\circle*{10}}
\put(290,4){\line (1,1){20}}
\put(312,27){\circle*{10}}
\put(312,27){\circle{16}}
\put(328,27){$\alpha_{p+r}$}
\put(290,-4){\line (1,-1){20}}
\put(312,-26){\circle*{10}}
\end{picture}\hskip 70pt\raisebox {8pt}{$D_{p+r}$}
}}
$$
\vskip 5pt
$\rm 1)$ If $r$ is odd and if $p=r-1$, this space is regular and $1$-irreducible    $($hence $Q$-irreducible from PropositionÊ $\ref{prop-irreductibilite-implications})$.

$2)$ If $r$ is odd and $p\leq r-2$, this space is not regular and there exists a non-trivial normal unipotent subgroup of the generic isotropy subgroup which is included in $SL(r)$.
 
 \end{lemma}
 \begin{proof} \hfill{}
 
 The space of the representation is $V=M(r,p)\oplus Skew(r)$, where $Skew(r)$ denotes the spaces of skew-symmetric  matrices of size $r$, and the action of the group $GL(p)\times GL(r)$ is given by
 $$(g_{1},g_{2})(X,Y)=  ({^t\hskip -2ptg_{2}^{-1}}Xg_{1}^{-1}, g_{2}Y  {^t\hskip -2ptg_{2}}),$$ where $g_{1}\in GL(p), g_{2}\in GL(r), X\in M(r,p), Y\in Skew(r)$.
 From the computations in \cite{Sato-Kimura}, p. 75-76, we know that if $r=2m+1$, there exists a generic element $Y_{0}\in Skew(r)$, such that the isotropy subgroup of $(0,Y_{0})\in V$ is the set of pairs of the form
 $$(g_{1}, \left[\hskip -5pt\begin{array}{cc}A&B\\
 0&D\end{array}\hskip -5pt\right])$$
 where $g_{1}\in GL(p)$, $A\in Sp(m)$, $B\in M(2m,1)$, $D\in GL(1)$, and where $Sp(m)$ denotes the symplectic group inside $GL(2m)$.
 \vskip 5pt
  $ \bullet$  {\it Suppose  that $p= r-1$}.
  
  One shows easily that if $X_{0}=\left[\hskip -5pt\begin{array}{c}
  I_{r-1}\\
  0\end{array}\hskip -5pt\right]\in M(r,r-1)$, the isotropy subgroup of $(X_{0},Y_{0})$ is the set of pairs of matrices of the form
  $$(g_{1}, \left[\hskip -5pt\begin{array}{cc}{^tg_{1}^{-1}}&0\\
 0&D\end{array}\hskip -5pt\right]),$$ 
 where $ g_{1}\in Sp(m), D\in GL(1)$.   
 
 A simple calculation of dimensions proves then that $(X_{0},Y_{0})$ is generic. As the preceding isotropy subgroup is reductive, this $PV$ is regular. The normal subgroup $G_{1}$ of $GL(r-1)\times GL(r)$ generated by this isotropy subgroup and the commutator subgroup $SL(r-1)\times SL(r)$ is of codimension one. Therefore this $PV$ is $1$-irreducible. The fundamental relative invariant is given by $f(X,Y)= Pf(^tX.Y.X)$  $(X\in M(r,r-1), Y\in Skew(r)$), where $Pf(Z) $ denotes the Pfaffian of the skew-symmetric matrix $Z$.

  \vskip 5pt
  $ \bullet$  {\it Suppose  that $p\leq r-2$}.
  
  Set $X_{0}=\left[\hskip -5pt\begin{array}{c}
  I_{p}\\
  0\end{array}\hskip -5pt\right]\in M(r,p)$. Then the isotropy subgroup of $(X_{0},Y_{0})$ is the set of pairs of matrices of the form
  
 $$(g_{1}, \left[\hskip -3pt\begin{array}{ccc}^tg_{1}^{-1}&0&0\\
 
 X&Y&B\\
 
 0&0&D
 \end{array}\hskip -3pt\right]),$$ 
where  $X\in M(r-1-p,p), Y\in GL(r-1-p), D\in GL(1), B\in M(r-1-p,1)$, and where  
$$\left[\hskip -3pt\begin{array}{cc}^tg_{1}^{-1}&0 \\
 
 X&Y

 \end{array}\hskip -3pt\right]\in Sp(m).$$
 Then the set of matrices $\left[\hskip -5pt\begin {array}{ccc}I_{p}&0&0\\
 0&I_{r-1-p}&B\\
 0&0&1\end{array}\hskip -5pt\right]$ is a normal unipotent subgroup in $SL(r)$.

  \end{proof}
  
  \begin{rem}\label{exemple1-invariant-nonreg}   If  $r$ is odd and $p$ is even ($p<r-2$) the function $(X,Y)\longmapsto Pf(^tX.Y.X)$ is a non-trivial relative invariant of the $PV$  considered in Lemma \ref{lemme-GLXGL-in-D(1)}, which is non regular for these values of $p$ and $r$. Hence the result from \cite{rub-note-PV} which asserts that an irreducible $PV$   of parabolic type is regular if and only if there exists a non-trivial relative invariant is no longer true if the representation is not irreducible.(See also Remark \ref{exemple2-invariant-nonreg} for another example).
  \end{rem}

 \begin{lemma}\label{lemme-GLXGL-in-D(2)} Let $D_{2}$ be the group $({{\bb C}^*})^2$ identified with the $2\times 2$ diagonal matrices, and denote by $\Delta$ the natural representation of $D_{2}$ on ${\bb C}^2$. Consider the representation
 $$[\Lambda_{p-1}(GL(p))\otimes \Lambda_{1}(SL(q))\otimes Id(D_{2})]\oplus [Id(GL(p))\otimes\Lambda_{q-1}(SL(q))\otimes \Delta]$$
 of the group $GL(p)\times SL(q)\times D_{2}$. Note that this representation is prehomogeneous since it is infinitesimally equivalent to the $PV$ of parabolic type associated to the diagram
 $$ {
\hbox{\unitlength=0.5pt
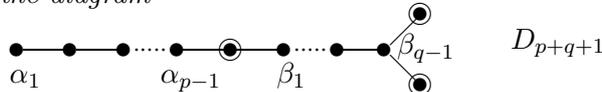
\begin{picture}(240,35)(0,-15)
\put(10,0){\circle*{10}}
\put(5,-25){$\alpha_1$}
\put(15,0){\line (1,0){30}}
\put(50,0){\circle*{10}}
\put(55,0){\line (1,0){30}}
\put(90,0){\circle*{10}}
\put(100,0){\circle*{1}}
\put(105,0){\circle*{1}}
\put(110,0){\circle*{1}}
\put(115,0){\circle*{1}}
\put(120,0){\circle*{1}}
\put(130,0){\circle*{10}}
\put(118,-25){$\alpha_{p-1}$}
\put(135,0){\line (1,0){30}}
\put(170,0){\circle*{10}}
\put(170,0){\circle{16}}
\put(175,0){\line (1,0){30}}
\put(210,0){\circle*{10}}
\put(205,-25){$\beta_1$}
\put(220,0){\circle*{1}}
\put(225,0){\circle*{1}}
\put(230,0){\circle*{1}}
\put(235,0){\circle*{1}}
\put(240,0){\circle*{1}}
\put(250,0){\circle*{10}}
\put(255,0){\line (1,0){30}}
\put(285,0){\circle*{10}}
\put(295,-3){$\beta_{q-1}$}
\put(290,4){\line (1,1){20}}
\put(312,27){\circle*{10}}
\put(312,27){\circle{16}}
 \put(290,-4){\line (1,-1){20}}
\put(312,-26){\circle*{10}}
\put(312,-26){\circle{16}}
\end{picture}\hskip 70pt\raisebox {8pt}{$D_{p+q+1}$}
}}
$$
\vskip 5pt
$1)$ If $q>p$ and $p=2$ this $PV$ is regular and $1$-irreducible  $($hence $Q$-irreducible from PropositionÊ $\ref{prop-irreductibilite-implications})$.

$2)$ If $q>p$ and $p\neq2$, then this $PV$ is not regular and there exists a non-trivial normal unipotent subgroup of the generic isotropy subgroup which is included in $SL(q)$.
 \end{lemma}
 
 \begin{proof}\hfill{}
 
  The space of the representation is $M(q,p)\oplus M(2,q)$ and the action of  $GL(p)\times SL(q)\times D_{2}$ is given by
  $$(g_{1},g_{2},g_{3})(X,Y)=(g_{2}Xg_{1}^{-1},g_{3}Yg_{2}^{-1}),$$
  where $g_{1}\in GL(p), g_{2}\in SL(q),g_{3}\in D_{2}, X\in M(q,p), Y\in M(2,q).$
  \vskip 5pt
  $ \bullet$  {\it Suppose  that $q>p$ and $p= 2$}.
  
  Let $X_{0}=\left[\hskip -5pt\begin{array}{c}I_{2}\\
  0
  \end{array}\hskip -5pt\right]\in M(q,2)$ and let $Y_{0}= \left[\hskip -5pt\begin{array}{cc}I_{2}&\hskip -5pt 0   \end{array}\hskip -5pt\right]\in M(2,q)$. A computation shows that the isotropy subgroup of $(X_{0},Y_{0})$ is the set of triplets of the form
  $(d, \left[\hskip -5pt\begin{array}{cc}d&0\\
  0&g
  \end{array}\hskip -5pt\right],d)$, where $d\in D_{2}$ and $g\in GL(q-2)$. From the dimensions of the full group and of the isotropy subgroup, we see that $(X_{0},Y_{0})$ is generic. Moreover as the isotropy subgroup is reductive, the $PV$ is regular. The subgroup $G_{1}$ generated by the commutator subgroup $(\simeq SL(2)\times SL(q))$ and the generic isotropy is the subgroup of triples $(g_{1},g_{2},g_{3})$  with $\det g_{1}=\det g_{3}$.  Hence $G/G_{1}$ is one dimensional, therefore the $PV$ is $1$-irreducible. It is easy to see that the function $(X,Y) \longmapsto \det (YX)$ is the fundamental relative invariant.
  
   \vskip 5pt
  $ \bullet$  {\it Suppose  that $q>p$ and $p> 2$}.

Let $X_{0}=\left[\hskip -5pt\begin{array}{c}I_{p}\\
  0
  \end{array}\hskip -5pt\right]\in M(q,p)$ and let $Y_{0}= \left[\hskip -5pt\begin{array}{cc}I_{2}&\hskip -5pt 0   \end{array}\hskip -5pt\right]\in M(2,q)$. Then again one proves that $(X_{0},Y_{0})$ is generic and one shows that its isotropy subgroup is the set of triples of the form
  $$(\left[\hskip -5pt\begin{array}{cc}d&0\\
  0&D
  \end{array}\hskip -5pt\right], \left[\hskip -5pt\begin{array}{ccc}d&0&0\\
  C&D&B\\
  0&0&D'
  \end{array}\hskip -5pt\right],d),$$
  where $d\in D_{2}, D\in GL(p-2), B\in M(p-2, q-p), D'\in GL(q-p)$.
  The set of matrices of the form $ \left[\hskip -5pt\begin{array}{ccc}I_{2}&0&0\\
  0&I_{p-2}&B\\
  0&0&I_{q-p}
  \end{array}\hskip -5pt\right]$ is a normal unipotent subgroup of $SL(q)$.
  
   \vskip 5pt
  $ \bullet$  {\it Suppose  that $q>p$ and $p=1$}.
  
  Let $X_{0}=\left[\hskip -5pt\begin{array}{c}1\\
  1\\
  0
  \end{array}\hskip -5pt\right]\in M(q,1)$ and let $Y_{0}= \left[\hskip -5pt\begin{array}{cc}I_{2}&\hskip -5pt 0   \end{array}\hskip -5pt\right]\in M(2,q)$.  It is easy to verify that $(X_{0},Y_{0})$ is generic and that its isotropy subgroup is the set of triples of the form
  $(\lambda,  \left[\hskip -5pt\begin{array}{ccc}\lambda&0&0\\
  0&\lambda&0\\
  \gamma&-\gamma&D
  \end{array}\hskip -5pt\right],\left[\hskip -5pt\begin{array}{cc}\lambda&0\\
  0&\lambda 
  \end{array}\hskip -5pt\right])$, where $\lambda\in {\bb C}^*$, $D\in GL(q-2),  \lambda^2\det D =1$, $\gamma\in M(q-2,1)$. The subset of matrices of the form
  $\left[\hskip -5pt\begin{array}{ccc}1&0&0\\
  0&1&0\\
  \gamma&-\gamma&I_{q-2}
  \end{array}\hskip -5pt\right]$ is a normal unipotent subgroup of $SL(q)$.

 \end{proof}
 
 \begin{rem}\label{exemple2-invariant-nonreg} If $\left[\hskip -5pt\begin{array}{c}x_{1}\\
  x_{2}
  \end{array}\hskip -5pt\right]$ is a vector in ${\bb C}^2$, let $f_{1}$ and $f_{2}$ be the two projections defined by $f_{i}(\left[\hskip -5pt\begin{array}{c}x_{1}\\
  x_{2}
  \end{array}\hskip -5pt\right])=x_{i}, i=1,2$. It is quite obvious that if $q>p$, and $p=1$, the mappings $(X,Y)\longmapsto f_{i}(Y.X)$ are relative invariants which are algebraically independant. This gives another example of a parabolic $PV$ having nontrivial relative invariants and which is nonregular (see Remark \ref{exemple1-invariant-nonreg}).
   \end{rem}
   
   \begin{lemma}\label{lemme-arete-gras} Let $({\go l}_{\theta},d_{1}(\theta))$ be a $PV$ of parabolic type in a simple Lie algebra ${\go g}$. Suppose that its diagram is of the following type
  $$ {
\hbox{\unitlength=0.5pt
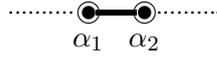
\begin{picture}(240,35)(0,-15)\put(72,0){\circle*{1}}
\put(77,0){\circle*{1}}
\put(82,0){\circle*{1}}
\put(87,0){\circle*{1}}
\put(92,0){\circle*{1}}
 \put(97,0){\circle*{1}}
\put(102,0){\circle*{1}}
\put(107,0){\circle*{1}}
\put(112,0){\circle*{1}}
\put(117,0){\circle*{1}}
\put(130,0){\circle*{10}}
\put(130,0){\circle{16}}
\put(118,-25){$\alpha_{1}$}
\put(136,-1){\line (1,0){30}}
\put(136,0){\line (1,0){30}}
\put(136,1){\line (1,0){30}}
\put(136,2){\line (1,0){30}}
\put(172,0){\circle*{10}}
\put(172,0){\circle{16}}
\put(160,-25){$\alpha_{2}$}
\put(183,0){\circle*{1}}
\put(188,0){\circle*{1}}
\put(193,0){\circle*{1}}
\put(198,0){\circle*{1}}
\put(203,0){\circle*{1}}
\put(208,0){\circle*{1}}
\put(213,0){\circle*{1}}
\put(218,0){\circle*{1}}
\put(223,0){\circle*{1}}
\put(227,0){\circle*{1}}
 
\end{picture} 
}}
$$
\vskip 5pt
where the boldface line  stands for one or more edges in the Dynkin diagram. In other words, in the notation of section $4.1$,  we suppose that $\Psi\setminus \theta$ contains two roots $\alpha_{1}$ and $\alpha_{2} $ $($but possibly others$)$  with $(\alpha_{1}|\alpha_{2})\neq 0$. Let $\Psi_{1}$ be the connected component of $\Psi\setminus \{\alpha_{2}\}$ containing $\alpha_{1}$ and let $\Psi_{2}$ be the connected component of $\Psi\setminus \{\alpha_{1}\}$ containing $\alpha_{2}$. Set $\theta_{1}=\theta\cap \Psi_{1}$ and $\theta_{2}=\theta\cap \Psi_{2}$. Define
$$D^1(\theta)=\bigoplus_{\alpha\in \Psi_{1}\setminus \theta_{1}}{\go g}^{\overline{\alpha}},\, \,D^2(\theta)=\bigoplus_{\alpha\in \Psi_{2}\setminus \theta_{2}}{\go g}^{\overline{\alpha}}$$
$($For the notations see section $4.1$, $D^1(\theta)$ $($resp.  $D^2(\theta)$$)$ is just the sum of the irreducible components of $d_{1}(\theta)$ arising from the left of the root $\alpha_{1}$ $($resp. from the right of the root $\alpha_{2}$$)$. Then:
$$(L_{\theta},d_{1}(\theta)) \text { is regular }\Longleftrightarrow (L_{\theta},D^1(\theta)) \text { and } (L_{\theta},D^2(\theta)) \text{ are regular. }$$
   \end{lemma}
 \begin{proof}\hfill{}

 Suppose first that $(L_{\theta},D^1(\theta))$ and $(L_{\theta},D^2(\theta))$ are regular. Then, as $d_{1}(\theta)= D^1(\theta)\oplus  D^2(\theta)$, we know from Proposition \ref{prop-somme-reg} that $(L_{\theta},d_1(\theta))$ is regular.
 
 Conversely suppose that $(L_{\theta},D^{1}(\theta))$ is not regular (for example). Let $X_{1}+X_{2}$ $(X_{i}\in D^i(\theta))$ a generic element in $d_{1}(\theta).$ Then $X_{1} $ is generic in $(L_{\theta},D^{1}(\theta))$ (see Proposition \ref{prop-components}). From the hypothesis we know that the isotropy subgroup $(L_{\theta})_{X_{1}}$ is not reductive (Proposition \ref{prop.PV-regulier-non-connexe}), hence $(L_{\theta})_{X_{1}}$ contains a nontrivial normal unipotent subgroup $U$. The Lie algebra ${\go u}$ of $U$ is a nonzero ideal in $({\go l}_{\theta})_{X_{1}}$. Let ${\go l}_{1}$ (resp ${\go l}_{2}$)  be the semi-simple subalgebra of ${\go g}$ corresponding to $\theta_{1}$ (resp. $\theta_{2}$). One has ${\go l}_{\theta}={\go h}_{\theta}\oplus {\go l}_{1}\oplus {\go l}_{2}$. From the hypothesis on $\alpha_{1} $ and  $\alpha_{2}$, we have $[{\go l}_{2}, X_{1}]=\{0\}$. Therefore $({\go l}_{\theta})_{X_{1}}=({\go h}_{\theta}\oplus {\go l}_{1})_{X_{1}}\oplus {\go l}_{2}$, and hence ${\go u}\in ({\go h}_{\theta}\oplus {\go l}_{1})_{X_{1}}$. But as ${\go u}$ is the Lie algebra of a unipotent subgroup we have  ${\go u}\in   {\go l}_{1} $. But as $[{\go l}_{1},X_{2}]=\{0\}$, we obtain that $U$ stabilizes also $X_{2}$. As $({\go l}_{\theta})_{X_{1}+X_{2}}=({\go l}_{\theta})_{X_{1}}\cap ({\go l}_{\theta})_{X_{2}}$, we see that ${\go u}$ is an ideal in $({\go l}_{\theta})_{X_{1}+X_{2}}$. Hence $U$ is a normal subgroup of $(L_{\theta})_{X_{1}+X_{2}}$. Therefore $(L_{\theta},d_{1}(\theta))$ is not regular.

 \end{proof}
 
 \begin{theorem}\label{th-Q-irred} The $Q$-irreducible $PV$'s of parabolic type which are not irreducible regular are exactly the $PV$'s from Table 1 at the end of the paper $($where the numbers $p_{i}$ are the number of roots in the connected components of $\theta$$)$.

  \end{theorem}
 
 \begin{proof}\hfill{}
 
 A consequence of Lemma \ref{lemme-arete-gras} is that the diagram of a $Q$-irreducible $PV$ of parabolic type will never contain two circled roots which are connected by one or more edges. Therefore we will never consider such diagrams in this proof.
 \vskip 3pt
 
$\diamondsuit$ {\it  Let us first consider the case of classical simple Lie algebras.}

 As we do not consider irreducible $PV$'s, we assume that $\text{Card}(\Psi\setminus \theta)\geq 2.$
 
 \vskip 5pt $\bullet$ The case $A_{n}$.

 a) Suppose that    $\text{Card}(\Psi\setminus \theta)= 2$. Consider a diagram of the  type:
 \vskip 3pt
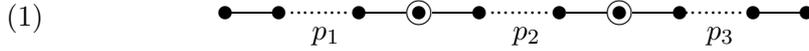
 $$\hskip -30pt \hbox to 9cm{\unitlength=0.5pt
  \begin{picture}(550,-20)
 \put(90,10){\circle*{10}}
 \put(95,10){\line (1,0){30}}
\put(130,10){\circle*{10}}
 
\put(140,10){\circle*{1}}
\put(145,10){\circle*{1}}
\put(150,10){\circle*{1}}
\put(155,10){\circle*{1}}
\put(160,10){\circle*{1}}
\put(165,10){\circle*{1}}
\put(170,10){\circle*{1}}
\put(175,10){\circle*{1}}
\put(180,10){\circle*{1}}
 \put(190,10){\circle*{10}}
 \put(155,-10){$p_{1}$}
 
\put(195,10){\line (1,0){30}}
\put(235,10){\circle*{10}}
 
\put(235,10){\circle{18}}
\put(245,10){\line (1,0){30}}
\put(280,10){\circle*{10}}
 
\put(290,10){\circle*{1}}
\put(295,10){\circle*{1}}
\put(300,10){\circle*{1}}
\put(305,10){\circle*{1}}
\put(310,10){\circle*{1}}
\put(315,10){\circle*{1}}
\put(320,10){\circle*{1}}
\put(325,10){\circle*{1}}
\put(330,10){\circle*{1}}
\put(305,-10){$p_{2}$}
\put(340,10){\circle*{10}}
 
\put(345,10){\line (1,0){30}}
\put(385,10){\circle*{10}}
\put(385,10){\circle{18}}
\put(395,10){\line (1,0){30}}
\put(430,10){\circle*{10}}

\put(435,10){\circle*{1}}
\put(440,10){\circle*{1}}
\put(445,10){\circle*{1}}
\put(450,10){\circle*{1}}
\put(455,10){\circle*{1}}
\put(460,10){\circle*{1}}
\put(465,10){\circle*{1}}
\put(470,10){\circle*{1}}
\put(475,10){\circle*{1}}
\put(485,10){\circle*{10}}
\put(490,10){\line (1,0){30}}
\put(525,10){\circle*{10}}
 \put(450,-10){$p_{3}$}
  \end{picture}
}\leqno {(1)}$$
  \vskip 3pt

 which is supposed to be Q-irreducible. If $p_{1}\geq p_{2}$, Lemma \ref{lemme-GXGL} implies that either the subdiagram 
 $$ \hbox to 9cm{\unitlength=0.5pt
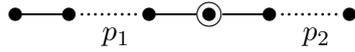
  \begin{picture}(550,-20)
 \put(90,10){\circle*{10}}
 \put(95,10){\line (1,0){30}}
\put(130,10){\circle*{10}}
 
\put(140,10){\circle*{1}}
\put(145,10){\circle*{1}}
\put(150,10){\circle*{1}}
\put(155,10){\circle*{1}}
\put(160,10){\circle*{1}}
\put(165,10){\circle*{1}}
\put(170,10){\circle*{1}}
\put(175,10){\circle*{1}}
\put(180,10){\circle*{1}}
 \put(190,10){\circle*{10}}
 \put(155,-10){$p_{1}$}
 
\put(195,10){\line (1,0){30}}
\put(235,10){\circle*{10}}
 
\put(235,10){\circle{18}}
\put(245,10){\line (1,0){30}}
\put(280,10){\circle*{10}}
 
\put(290,10){\circle*{1}}
\put(295,10){\circle*{1}}
\put(300,10){\circle*{1}}
\put(305,10){\circle*{1}}
\put(310,10){\circle*{1}}
\put(315,10){\circle*{1}}
\put(320,10){\circle*{1}}
\put(325,10){\circle*{1}}
\put(330,10){\circle*{1}}
\put(305,-10){$p_{2}$}
\put(340,10){\circle*{10}}
 
   \end{picture}
  }
$$
is either regular or  the generic isotropy subgroup contains a nontrivial unipotent subgroup which is included in $SL(p_{1}+1)$. Therefore, in the second case,  this unipotent subgroup  will be included in the generic isotropy of the diagram $(1)$. Hence, in the second case the diagram $(1)$,  will not be regular.
Therefore we have necessarily $p_{1}< p_{2}$. The same arguments show that we have also $p_{3}<p_{2}.$ But then, from  Lemma \ref{lemme-GLXGLXG} we obtain that this $PV$ is regular if and only if $p_{3}=p_{2}$, and in this case it is $1$-irreducible.
\vskip 5pt 
b) Suppose that    $\text{Card}(\Psi\setminus \theta)> 2$. Suppose that the following diagram is $Q$-irreducible:
 \vskip 3pt
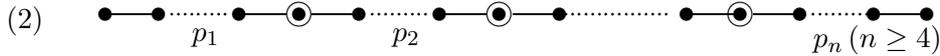
 $$ \hskip -120pt \hbox to 9cm{\unitlength=0.5pt
  \begin{picture}(550,-20)
 \put(90,10){\circle*{10}}
 \put(95,10){\line (1,0){30}}
\put(130,10){\circle*{10}}
 
\put(140,10){\circle*{1}}
\put(145,10){\circle*{1}}
\put(150,10){\circle*{1}}
\put(155,10){\circle*{1}}
\put(160,10){\circle*{1}}
\put(165,10){\circle*{1}}
\put(170,10){\circle*{1}}
\put(175,10){\circle*{1}}
\put(180,10){\circle*{1}}
 \put(190,10){\circle*{10}}
 \put(155,-10){$p_{1}$}
 
\put(195,10){\line (1,0){30}}
\put(235,10){\circle*{10}}
 
\put(235,10){\circle{18}}
\put(245,10){\line (1,0){30}}
\put(280,10){\circle*{10}}
 
\put(290,10){\circle*{1}}
\put(295,10){\circle*{1}}
\put(300,10){\circle*{1}}
\put(305,10){\circle*{1}}
\put(310,10){\circle*{1}}
\put(315,10){\circle*{1}}
\put(320,10){\circle*{1}}
\put(325,10){\circle*{1}}
\put(330,10){\circle*{1}}
\put(305,-10){$p_{2}$}
\put(340,10){\circle*{10}}
 
\put(345,10){\line (1,0){30}}
\put(385,10){\circle*{10}}
\put(385,10){\circle{18}}
\put(395,10){\line (1,0){30}}
\put(430,10){\circle*{10}}

\put(435,10){\circle*{1}}
\put(440,10){\circle*{1}}
\put(445,10){\circle*{1}}
\put(450,10){\circle*{1}}
\put(455,10){\circle*{1}}
\put(460,10){\circle*{1}}
\put(465,10){\circle*{1}}
\put(470,10){\circle*{1}}
\put(475,10){\circle*{1}}
\put(480,10){\circle*{1}}
\put(485,10){\circle*{1}}
\put(490,10){\circle*{1}}
\put(495,10){\circle*{1}}
\put(500,10){\circle*{1}}
\put(505,10){\circle*{1}}
\put(510,10){\circle*{1}}

\put(525,10){\circle*{10}}
\put(530,10){\line (1,0){30}}
\put(565,10){\circle*{10}}
\put(565,10){\circle{18}}
\put(575,10){\line (1,0){30}}
\put(610,10){\circle*{10}}
\put(620,10){\circle*{1}}
\put(625,10){\circle*{1}}
\put(630,10){\circle*{1}}
\put(635,10){\circle*{1}}
\put(640,10){\circle*{1}}
\put(645,10){\circle*{1}}
\put(650,10){\circle*{1}}
\put(655,10){\circle*{1}}
\put(660,10){\circle*{1}}
\put(665,10){\circle*{10}}
\put(670,10){\line (1,0){30}}
\put(705,10){\circle*{10}}
\put(620,-15){$p_{n}\,(n\geq 4)$}

\end{picture}
}\leqno {(2)}$$
  \vskip 3pt
  As before  Lemma \ref{lemme-GXGL} implies that $p_{1}<p_{2}$ and $p_{n}< p_{n-1}$. By induction the same argument shows that there exists $i\in \{2,\dots,n-1\}$ such that $p_{i-1}<p_{i}$ and $p_{i+1}<p_{i}$. If $p_{i-1}\neq p_{i+1}$  Lemma \ref{lemme-GLXGLXG} implies that there exists a normal unipotent subgroup in the generic isotropy subgroup of the subdiagram
   \vskip 3pt
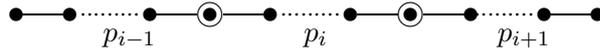
 $$\hskip -30pt \hbox to 9cm{\unitlength=0.5pt
  \begin{picture}(550,-20)
 \put(90,10){\circle*{10}}
 \put(95,10){\line (1,0){30}}
\put(130,10){\circle*{10}}
 
\put(140,10){\circle*{1}}
\put(145,10){\circle*{1}}
\put(150,10){\circle*{1}}
\put(155,10){\circle*{1}}
\put(160,10){\circle*{1}}
\put(165,10){\circle*{1}}
\put(170,10){\circle*{1}}
\put(175,10){\circle*{1}}
\put(180,10){\circle*{1}}
 \put(190,10){\circle*{10}}
 \put(155,-10){$p_{i-1}$}
 
\put(195,10){\line (1,0){30}}
\put(235,10){\circle*{10}}
 
\put(235,10){\circle{18}}
\put(245,10){\line (1,0){30}}
\put(280,10){\circle*{10}}
 
\put(290,10){\circle*{1}}
\put(295,10){\circle*{1}}
\put(300,10){\circle*{1}}
\put(305,10){\circle*{1}}
\put(310,10){\circle*{1}}
\put(315,10){\circle*{1}}
\put(320,10){\circle*{1}}
\put(325,10){\circle*{1}}
\put(330,10){\circle*{1}}
\put(305,-10){$p_{i}$}
\put(340,10){\circle*{10}}
 
\put(345,10){\line (1,0){30}}
\put(385,10){\circle*{10}}
\put(385,10){\circle{18}}
\put(395,10){\line (1,0){30}}
\put(430,10){\circle*{10}}

\put(435,10){\circle*{1}}
\put(440,10){\circle*{1}}
\put(445,10){\circle*{1}}
\put(450,10){\circle*{1}}
\put(455,10){\circle*{1}}
\put(460,10){\circle*{1}}
\put(465,10){\circle*{1}}
\put(470,10){\circle*{1}}
\put(475,10){\circle*{1}}
\put(485,10){\circle*{10}}
\put(490,10){\line (1,0){30}}
\put(525,10){\circle*{10}}
 \put(450,-10){$p_{i+1}$}
  \end{picture}
}$$
  \vskip 7pt
which is included in $SL(p_{i})$. But this subgroup will still be included in the generic isotropy of the diagram $(2)$, and hence the diagram $(2)$ would not be regular.

If $p_{i-1}=p_{i+1}$ Lemma \ref{lemme-GLXGLXG} implies that the subdiagram above is regular, hence diagram $(2)$ is never $Q$-irreducible.

 \vskip 5pt
  $\bullet$ The case $B_{n}$.
  
   a) Suppose that    $\text{Card}(\Psi\setminus \theta)= 2$. Suppose that the diagram

  $$ \hbox to 9,5cm {\unitlength=0.5pt
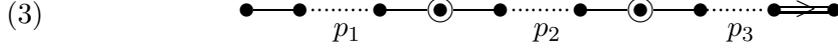
  \begin{picture}(550,-20)
 \put(90,10){\circle*{10}}
 \put(95,10){\line (1,0){30}}
\put(130,10){\circle*{10}}
 
\put(140,10){\circle*{1}}
\put(145,10){\circle*{1}}
\put(150,10){\circle*{1}}
\put(155,10){\circle*{1}}
\put(160,10){\circle*{1}}
\put(165,10){\circle*{1}}
\put(170,10){\circle*{1}}
\put(175,10){\circle*{1}}
\put(180,10){\circle*{1}}
 \put(190,10){\circle*{10}}
 \put(155,-10){$p_{1}$}
 
\put(195,10){\line (1,0){30}}
\put(235,10){\circle*{10}}
 
\put(235,10){\circle{18}}
\put(245,10){\line (1,0){30}}
\put(280,10){\circle*{10}}
 
\put(290,10){\circle*{1}}
\put(295,10){\circle*{1}}
\put(300,10){\circle*{1}}
\put(305,10){\circle*{1}}
\put(310,10){\circle*{1}}
\put(315,10){\circle*{1}}
\put(320,10){\circle*{1}}
\put(325,10){\circle*{1}}
\put(330,10){\circle*{1}}
\put(305,-10){$p_{2}$}
\put(340,10){\circle*{10}}
 
\put(345,10){\line (1,0){30}}
\put(385,10){\circle*{10}}
\put(385,10){\circle{18}}
\put(395,10){\line (1,0){30}}
\put(430,10){\circle*{10}}

\put(435,10){\circle*{1}}
\put(440,10){\circle*{1}}
\put(445,10){\circle*{1}}
\put(450,10){\circle*{1}}
\put(455,10){\circle*{1}}
\put(460,10){\circle*{1}}
\put(465,10){\circle*{1}}
\put(470,10){\circle*{1}}
\put(475,10){\circle*{1}}
\put(485,10){\circle*{10}}
\put(484,12){\line (1,0){41}}
\put(484,8){\line (1,0){41}}
\put(500,5.5){$>$}
\put(530,10){\circle*{10}}
 
 \put(450,-10){$p_{3}$}
 
 \end{picture}
}\leqno (3)$$
\vskip 3pt
is $Q$-irreducible. As in the case $A_{n} \, a)$ before, Lemma \ref{lemme-GXGL} implies that $p_{1}<p_{2}$ and $2p_{3}+1< p_{2}+1$. Then Lemma \ref{lemme-GLXGLXG} implies that the diagram $(3)$ is $Q$-irreducible if and only if $2p_{3}+1=p_{1}+1$, which is the  condition in Table $1$.

   b) Suppose that    $\text{Card}(\Psi\setminus \theta)> 2$.  Suppose that the following diagram is $Q$-irreducible:
 \vskip 3pt
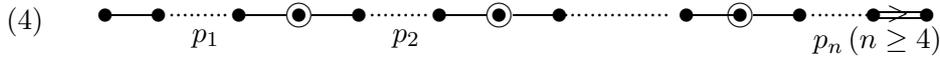
 $$ \hskip -120pt \hbox to 9cm{\unitlength=0.5pt
  \begin{picture}(550,-20)
 \put(90,10){\circle*{10}}
 \put(95,10){\line (1,0){30}}
\put(130,10){\circle*{10}}
 
\put(140,10){\circle*{1}}
\put(145,10){\circle*{1}}
\put(150,10){\circle*{1}}
\put(155,10){\circle*{1}}
\put(160,10){\circle*{1}}
\put(165,10){\circle*{1}}
\put(170,10){\circle*{1}}
\put(175,10){\circle*{1}}
\put(180,10){\circle*{1}}
 \put(190,10){\circle*{10}}
 \put(155,-10){$p_{1}$}
 
\put(195,10){\line (1,0){30}}
\put(235,10){\circle*{10}}
 
\put(235,10){\circle{18}}
\put(245,10){\line (1,0){30}}
\put(280,10){\circle*{10}}
 
\put(290,10){\circle*{1}}
\put(295,10){\circle*{1}}
\put(300,10){\circle*{1}}
\put(305,10){\circle*{1}}
\put(310,10){\circle*{1}}
\put(315,10){\circle*{1}}
\put(320,10){\circle*{1}}
\put(325,10){\circle*{1}}
\put(330,10){\circle*{1}}
\put(305,-10){$p_{2}$}
\put(340,10){\circle*{10}}
 
\put(345,10){\line (1,0){30}}
\put(385,10){\circle*{10}}
\put(385,10){\circle{18}}
\put(395,10){\line (1,0){30}}
\put(430,10){\circle*{10}}

\put(435,10){\circle*{1}}
\put(440,10){\circle*{1}}
\put(445,10){\circle*{1}}
\put(450,10){\circle*{1}}
\put(455,10){\circle*{1}}
\put(460,10){\circle*{1}}
\put(465,10){\circle*{1}}
\put(470,10){\circle*{1}}
\put(475,10){\circle*{1}}
\put(480,10){\circle*{1}}
\put(485,10){\circle*{1}}
\put(490,10){\circle*{1}}
\put(495,10){\circle*{1}}
\put(500,10){\circle*{1}}
\put(505,10){\circle*{1}}
\put(510,10){\circle*{1}}

\put(525,10){\circle*{10}}
\put(530,10){\line (1,0){30}}
\put(565,10){\circle*{10}}
\put(565,10){\circle{18}}
\put(575,10){\line (1,0){30}}
\put(610,10){\circle*{10}}
\put(620,10){\circle*{1}}
\put(625,10){\circle*{1}}
\put(630,10){\circle*{1}}
\put(635,10){\circle*{1}}
\put(640,10){\circle*{1}}
\put(645,10){\circle*{1}}
\put(650,10){\circle*{1}}
\put(655,10){\circle*{1}}
\put(660,10){\circle*{1}}
\put(665,10){\circle*{10}}
\put(664,12){\line (1,0){41}}
\put(664,8){\line (1,0){41}}
\put(675,5.5){$>$}

\put(705,10){\circle*{10}}
\put(620,-15){$p_{n}\,(n\geq 4)$}

\end{picture}
}\leqno {(4)}$$
  \vskip 3pt
  Then as before Lemma \ref{lemme-GXGL} implies that $p_{1}<p_{2}$ and $2p_{n}<p_{n-1}$. There are then two possibilities:
  
  \hskip 15pt - either there exists $i\in \{2,\dots,n-2\}$ such that $p_{i-1}<p_{i}$ and $p_{i+1}<p_{i},$
  
   \hskip 15pt - or $p_{n-2}<p_{n-1}$ and $2p_{n}< p_{n-1}$.
   
   In both cases Lemma \ref{lemme-GLXGLXG} implies that either   diagramm $(4)$ contains a regular subdiagram or it is not regular. We have showed that   diagram $(4)$ is never Q-irreducible.

  \vskip 5pt
  $\bullet$ The cases $C_{n}$ and $D_{n}^{1}$.

  These cases   can be treated in the same way as the cases $A_{n}$ and  $B_{n} $. It must be noticed that in the $C_{n}$ case one cannot have a diagram where the root $\alpha_{n}$ is circled. This is because the subdiagram 
$$  
 \hbox{\unitlength=0.5pt
\begin{picture}(300,30)(-82,0)
\put(10,10){\circle*{10}}

\put(15,10){\line (1,0){30}}
\put(50,10){\circle*{10}}
\put(55,10){\line (1,0){30}}
\put(90,10){\circle*{10}}
\put(95,10){\circle*{1}}
\put(100,10){\circle*{1}}
\put(105,10){\circle*{1}}
\put(110,10){\circle*{1}}
\put(115,10){\circle*{1}}
\put(120,10){\circle*{1}}
\put(125,10){\circle*{1}}
\put(130,10){\circle*{1}}
\put(135,10){\circle*{1}}
\put(140,10){\circle*{10}}
\put(145,10){\line (1,0){30}}
\put(180,10){\circle*{10}}
\put(184,12){\line (1,0){41}}
\put(184,8){\line(1,0){41}}
\put(190,5.5){$<$}
\put(220,10){\circle*{10}}
\put(220,10){\circle{18}}

\end{picture}
}
$$
would be regular (see  the list of the irreductible regular $PV$'s of parabolic type in \cite{rub-kyoto} or in \cite{rub-bouquin-PV}).
  
  \vskip 5pt
  $\bullet$ The case   $D_{n}^{2}$.   
  
      \vskip 3pt
   a) Suppose that    $\text{Card}(\Psi\setminus \theta)= 2$. Suppose that the diagram

$$ \hskip -40pt\hbox to 9cm {\unitlength=0.5pt
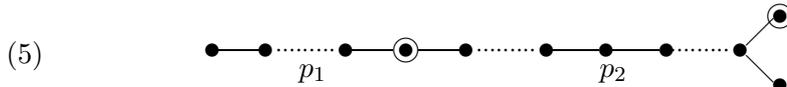
  \begin{picture}(550,-20)
 \put(90,10){\circle*{10}}
 \put(95,10){\line (1,0){30}}
\put(130,10){\circle*{10}}
 
\put(140,10){\circle*{1}}
\put(145,10){\circle*{1}}
\put(150,10){\circle*{1}}
\put(155,10){\circle*{1}}
\put(160,10){\circle*{1}}
\put(165,10){\circle*{1}}
\put(170,10){\circle*{1}}
\put(175,10){\circle*{1}}
\put(180,10){\circle*{1}}
 \put(190,10){\circle*{10}}
 \put(155,-10){$p_{1}$}
 
\put(195,10){\line (1,0){30}}
\put(235,10){\circle*{10}}
 
\put(235,10){\circle{18}}
\put(245,10){\line (1,0){30}}
\put(280,10){\circle*{10}}
 
\put(290,10){\circle*{1}}
\put(295,10){\circle*{1}}
\put(300,10){\circle*{1}}
\put(305,10){\circle*{1}}
\put(310,10){\circle*{1}}
\put(315,10){\circle*{1}}
\put(320,10){\circle*{1}}
\put(325,10){\circle*{1}}
\put(330,10){\circle*{1}}
\put(380,-10){$p_{2}$}
\put(340,10){\circle*{10}}
 
\put(345,10){\line (1,0){34}}
\put(385,10){\circle*{10}}
 
\put(390,10){\line (1,0){34}}
\put(430,10){\circle*{10}}

\put(435,10){\circle*{1}}
\put(440,10){\circle*{1}}
\put(445,10){\circle*{1}}
\put(450,10){\circle*{1}}
\put(455,10){\circle*{1}}
\put(460,10){\circle*{1}}
\put(465,10){\circle*{1}}
\put(470,10){\circle*{1}}
\put(475,10){\circle*{1}}
\put(485,10){\circle*{10}}
\put(490,14){\line (1,1){20}}
\put(515,36){\circle*{10}}
\put(515,36){\circle{18}}
\put(490,6){\line (1,-1){20}}
\put(515,-16){\circle*{10}}
\end{picture}
}\leqno (5)$$
is $Q$-irreducible. Then $p_{2}$ is even because if $p_{2}$ is   odd the subdiagram 
\vskip 3pt

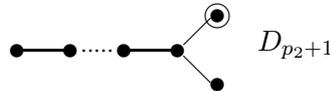
$$\hskip -80 pt \hbox{\unitlength=0.5pt
\begin{picture}(300,40)(-82,-10)
 \put(90,10){\circle*{10}}
\put(95,10){\line (1,0){30}}
\put(130,10){\circle*{10}}
\put(140,10){\circle*{1}}
\put(145,10){\circle*{1}}
\put(150,10){\circle*{1}}
\put(155,10){\circle*{1}}
\put(160,10){\circle*{1}}
\put(170,10){\circle*{10}}
\put(175,10){\line (1,0){30}}
\put(210,10){\circle*{10}}
\put(215,14){\line (1,1){20}}
\put(240,36){\circle*{10}}
\put(240,36){\circle{18}}
\put(215,6){\line (1,-1){20}}
\put(240,-16){\circle*{10}}

\put(270,10){$D_{p_{2}+1}$}

 \end{picture}
}
$$
would be regular  (see  the list of the irreductible regular $PV$'s of parabolic type in \cite{rub-kyoto} or in \cite{rub-bouquin-PV}). 

On the other hand from Lemma \ref{lemme-GXGL} we get that $p_{2}>p_{1}$. Then Lemma \ref{lemme-GLXGL-in-D(1)} implies that only the case where $p_{1}=p_{2}-1$ corresponds to a $Q$-irreducible $PV$.

b) Suppose that    $\text{Card}(\Psi\setminus \theta)> 2$.  Suppose that the following diagram is $Q$-irreducible:
 \vskip 3pt
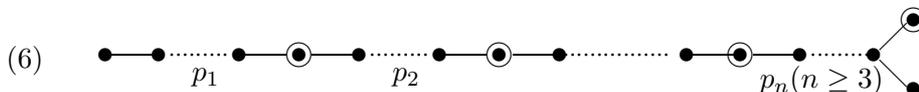
 $$ \hskip -120pt \hbox to 9cm{\unitlength=0.5pt
  \begin{picture}(550,-20)
 \put(90,10){\circle*{10}}
 \put(95,10){\line (1,0){30}}
\put(130,10){\circle*{10}}
 
\put(140,10){\circle*{1}}
\put(145,10){\circle*{1}}
\put(150,10){\circle*{1}}
\put(155,10){\circle*{1}}
\put(160,10){\circle*{1}}
\put(165,10){\circle*{1}}
\put(170,10){\circle*{1}}
\put(175,10){\circle*{1}}
\put(180,10){\circle*{1}}
 \put(190,10){\circle*{10}}
 \put(155,-10){$p_{1}$}
 
\put(195,10){\line (1,0){30}}
\put(235,10){\circle*{10}}
 
\put(235,10){\circle{18}}
\put(245,10){\line (1,0){30}}
\put(280,10){\circle*{10}}
 
\put(290,10){\circle*{1}}
\put(295,10){\circle*{1}}
\put(300,10){\circle*{1}}
\put(305,10){\circle*{1}}
\put(310,10){\circle*{1}}
\put(315,10){\circle*{1}}
\put(320,10){\circle*{1}}
\put(325,10){\circle*{1}}
\put(330,10){\circle*{1}}
\put(305,-10){$p_{2}$}
\put(340,10){\circle*{10}}
 
\put(345,10){\line (1,0){30}}
\put(385,10){\circle*{10}}
\put(385,10){\circle{18}}
\put(395,10){\line (1,0){30}}
\put(430,10){\circle*{10}}

\put(435,10){\circle*{1}}
\put(440,10){\circle*{1}}
\put(445,10){\circle*{1}}
\put(450,10){\circle*{1}}
\put(455,10){\circle*{1}}
\put(460,10){\circle*{1}}
\put(465,10){\circle*{1}}
\put(470,10){\circle*{1}}
\put(475,10){\circle*{1}}
\put(480,10){\circle*{1}}
\put(485,10){\circle*{1}}
\put(490,10){\circle*{1}}
\put(495,10){\circle*{1}}
\put(500,10){\circle*{1}}
\put(505,10){\circle*{1}}
\put(510,10){\circle*{1}}

\put(525,10){\circle*{10}}
\put(530,10){\line (1,0){30}}
\put(565,10){\circle*{10}}
\put(565,10){\circle{18}}
\put(575,10){\line (1,0){30}}
\put(610,10){\circle*{10}}
\put(620,10){\circle*{1}}
\put(625,10){\circle*{1}}
\put(630,10){\circle*{1}}
\put(635,10){\circle*{1}}
\put(640,10){\circle*{1}}
\put(645,10){\circle*{1}}
\put(650,10){\circle*{1}}
\put(655,10){\circle*{1}}
\put(660,10){\circle*{1}}
\put(665,10){\circle*{10}}
 \put(670,14){\line (1,1){20}}
\put(695,36){\circle*{10}}
\put(695,36){\circle{18}}
\put(670,6){\line (1,-1){20}}
\put(695,-16){\circle*{10}}

\put(580,-15){$p_{n}(n\geq 3)$}

\end{picture}
}\leqno {(6)}$$
  \vskip 7pt
  For the same reason as for the diagram $(5)$, we necessarily have $p_{n}$ even.
  Then from Lemma \ref{lemme-GXGL} we get $p_{1}<p_{2} $ and from Lemma \ref{lemme-GLXGL-in-D(1)} we get $p_{n}\leq p_{n-1}$. If $p_{n}=p_{n-1}$  diagram $(6)$ would contain the regular subdiagram 
   \vskip 3pt
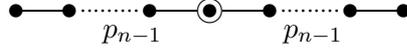
 $$ \hskip -10pt \hbox to 9cm{\unitlength=0.5pt
  \begin{picture}(550,-20)
 \put(90,10){\circle*{10}}
 \put(95,10){\line (1,0){30}}
\put(130,10){\circle*{10}}
 
\put(140,10){\circle*{1}}
\put(145,10){\circle*{1}}
\put(150,10){\circle*{1}}
\put(155,10){\circle*{1}}
\put(160,10){\circle*{1}}
\put(165,10){\circle*{1}}
\put(170,10){\circle*{1}}
\put(175,10){\circle*{1}}
\put(180,10){\circle*{1}}
 \put(190,10){\circle*{10}}
 \put(155,-10){$p_{n-1}$}
 
\put(195,10){\line (1,0){30}}
\put(235,10){\circle*{10}}
 
\put(235,10){\circle{18}}
\put(245,10){\line (1,0){30}}
\put(280,10){\circle*{10}}
 
\put(290,10){\circle*{1}}
\put(295,10){\circle*{1}}
\put(300,10){\circle*{1}}
\put(305,10){\circle*{1}}
\put(310,10){\circle*{1}}
\put(315,10){\circle*{1}}
\put(320,10){\circle*{1}}
\put(325,10){\circle*{1}}
\put(330,10){\circle*{1}}
\put(290,-10){$p_{n-1}$}
\put(340,10){\circle*{10}}
\put(345,10){\line (1,0){30}}
\put(380,10){\circle*{10}}

\end{picture}
}$$
\vskip 5pt
Hence $p_{1}<p_{2}$ and $p_{n}<p_{n-1}$. There exists then $i\in \{2,\dots,n\}$ such that $p_{i-1}<p_{i}$ and $p_{i+1}<p_{i}$. From Lemma \ref{lemme-GLXGLXG} we obtain that either the diagram $(6)$ is not $Q$-irreducible (if $p_{i-1}=p_{i+1}$), or non regular (if $p_{i-1}\neq p_{i+1}$). In any case  diagram $(6)$ is never $Q$-irreducible.

\vskip 5pt
  $\bullet$ The case   $D_{n}^{3}$.   
  
  a) Suppose that    $\text{Card}(\Psi\setminus \theta)= 2$. It is easy to prove that the subdiagram 
\vskip 3pt

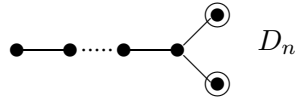
$$\hskip -80 pt \hbox{\unitlength=0.5pt
\begin{picture}(300,40)(-82,-10)
 \put(90,10){\circle*{10}}
\put(95,10){\line (1,0){30}}
\put(130,10){\circle*{10}}
\put(140,10){\circle*{1}}
\put(145,10){\circle*{1}}
\put(150,10){\circle*{1}}
\put(155,10){\circle*{1}}
\put(160,10){\circle*{1}}
\put(170,10){\circle*{10}}
\put(175,10){\line (1,0){30}}
\put(210,10){\circle*{10}}
\put(215,14){\line (1,1){20}}
\put(240,36){\circle*{10}}
\put(240,36){\circle{18}}
\put(215,6){\line (1,-1){20}}
\put(240,-16){\circle*{10}}
\put(240,-16){\circle{18}}

\put(270,10){$D_{n}$}

 \end{picture}
}
$$
 is regular  if and only if $n=3$ , and then $D_{3}=A_{3}$ and the corresponding diagram was already considered in the $A_{n}$ case.
 \vskip 3pt 
 
   b) Suppose that    $\text{Card}(\Psi\setminus \theta)= 3$. Suppose that the following diagram is $Q$-irreducible.
    \vskip 3pt
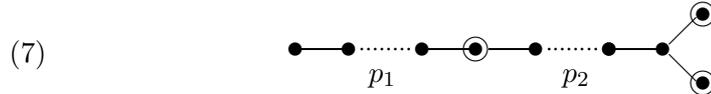
 $$ \hskip -320pt \hbox to 9cm{\unitlength=0.5pt
  \begin{picture}(550,-20)
   
 \put(430,10){\circle*{10}}
 \put(435,10){\line (1,0){30}}
\put(470,10){\circle*{10}}

 \put(480,10){\circle*{1}}
\put(485,10){\circle*{1}}
\put(485,-15){$p_{1}$}
\put(490,10){\circle*{1}}
\put(495,10){\circle*{1}}
\put(500,10){\circle*{1}}
\put(505,10){\circle*{1}}
\put(510,10){\circle*{1}}
\put(515,10){\circle*{1}}

\put(525,10){\circle*{10}}
\put(530,10){\line (1,0){30}}
\put(565,10){\circle*{10}}
\put(565,10){\circle{18}}
\put(575,10){\line (1,0){30}}
\put(610,10){\circle*{10}}
\put(620,10){\circle*{1}}
\put(625,10){\circle*{1}}
 \put(630,10){\circle*{1}}
\put(635,10){\circle*{1}}
\put(640,10){\circle*{1}}
\put(645,10){\circle*{1}}
\put(650,10){\circle*{1}}
\put(655,10){\circle*{1}}
\put(665,10){\circle*{10}}
\put(670,10){\line (1,0){30}}
\put(705,10){\circle*{10}}

\put(665,10){\circle*{10}}
 \put(710,14){\line (1,1){20}}
\put(735,36){\circle*{10}}
\put(735,36){\circle{18}}
\put(710,6){\line (1,-1){20}}
\put(735,-16){\circle*{10}}
\put(735,-16){\circle{18}}

\put(630,-15){$p_{2}$}

\end{picture}
}\leqno {(7)}$$
\vskip 3pt
We know from Lemma \ref{lemme-GLXGL-in-D(2)} that if $p_{2}>p_{1}$ and $p_{1}\neq 1$, diagram $(7)$ is not regular. If $p_{2}>p_{1}$ and $p_{1}=1$, the same Lemma implies that diagram $(7)$ is $Q$-irreducible.

If $p_{1}=p_{2}$, diagram $(7)$ contains obviously an $A_{n-2}$ regular irreducible subdiagram.

If $p_{1}>p_{2}$, diagram $(7)$ cannot be regular, as shown by Lemma \ref{lemme-GXGL}.

 \vskip 3pt 
 
   c) Suppose that    $\text{Card}(\Psi\setminus \theta)> 3$. The corresponding diagram is the following:
    \vskip 3pt
 $$ \hskip -120pt \hbox to 9cm{\unitlength=0.5pt
  \begin{picture}(550,-20)
 \put(90,10){\circle*{10}}
 \put(95,10){\line (1,0){30}}
\put(130,10){\circle*{10}}
 
\put(140,10){\circle*{1}}
\put(145,10){\circle*{1}}
\put(150,10){\circle*{1}}
\put(155,10){\circle*{1}}
\put(160,10){\circle*{1}}
\put(165,10){\circle*{1}}
\put(170,10){\circle*{1}}
\put(175,10){\circle*{1}}
\put(180,10){\circle*{1}}
 \put(190,10){\circle*{10}}
 \put(155,-10){$p_{1}$}
 
\put(195,10){\line (1,0){30}}
\put(235,10){\circle*{10}}
 
\put(235,10){\circle{18}}
\put(245,10){\line (1,0){30}}
\put(280,10){\circle*{10}}
 
\put(290,10){\circle*{1}}
\put(295,10){\circle*{1}}
\put(300,10){\circle*{1}}
\put(305,10){\circle*{1}}
\put(310,10){\circle*{1}}
\put(315,10){\circle*{1}}
\put(320,10){\circle*{1}}
\put(325,10){\circle*{1}}
\put(330,10){\circle*{1}}
\put(305,-10){$p_{2}$}
\put(340,10){\circle*{10}}
 
\put(345,10){\line (1,0){30}}
\put(385,10){\circle*{10}}
\put(385,10){\circle{18}}
\put(395,10){\line (1,0){30}}
\put(430,10){\circle*{10}}

\put(435,10){\circle*{1}}
\put(440,10){\circle*{1}}
\put(445,10){\circle*{1}}
\put(450,10){\circle*{1}}
\put(455,10){\circle*{1}}
\put(460,10){\circle*{1}}
\put(465,10){\circle*{1}}
\put(470,10){\circle*{1}}
\put(475,10){\circle*{1}}
\put(480,10){\circle*{1}}
\put(485,10){\circle*{1}}
\put(490,10){\circle*{1}}
\put(495,10){\circle*{1}}
\put(500,10){\circle*{1}}
\put(505,10){\circle*{1}}
\put(510,10){\circle*{1}}

\put(525,10){\circle*{10}}
\put(530,10){\line (1,0){30}}
\put(565,10){\circle*{10}}
\put(565,10){\circle{18}}
\put(575,10){\line (1,0){30}}
\put(610,10){\circle*{10}}
\put(620,10){\circle*{1}}
\put(625,10){\circle*{1}}
\put(630,10){\circle*{1}}
\put(635,10){\circle*{1}}
\put(640,10){\circle*{1}}
\put(645,10){\circle*{1}}
\put(650,10){\circle*{1}}
\put(655,10){\circle*{1}}
\put(660,10){\circle*{1}}
\put(665,10){\circle*{10}}
 \put(670,14){\line (1,1){20}}
\put(695,36){\circle*{10}}
\put(695,36){\circle{18}}
\put(670,6){\line (1,-1){20}}
\put(695,-16){\circle*{10}}
\put(695,-16){\circle{18}}

\put(580,-15){$p_{n}(n\geq 3)$}

\end{picture}
}\leqno {(8)}$$
  \vskip 7pt
  From  Lemma \ref{lemme-GXGL} and Lemma \ref{lemme-GLXGL-in-D(2)} we deduce that if this diagram would be $Q$-irreducible, we would have $p_{1}<p_{2}$ and $p_{n}<p_{n-1}$. Then, using the same method as in the $A_{n}$ case, one proves that diagram $(8)$ is never $Q$-irreducible.
  \vskip 3pt 
  $\diamondsuit$ {\it  Let us now  consider the case of exceptional  simple Lie algebras.}
  
  We only give the proof for $E_{6}$. The cases of $E_{7}$, $E_{8}$, $F_{4}$ and $G_{2}$ are analogous.
  
  We begin by writing  down all  possible diagrams in which  at least two roots are circled. The only (important) constraint comes from Lemma \ref{lemme-arete-gras} which excludes diagrams having two circled roots which are connected. If a diagram contains a regular subdiagram, we will write the subdiagram on the same line. Taking into account the symmetry of the Dynkin diagram of $E_{6}$, the list is as follows:
 
  $\begin{array}{clcl}
  1)& \hbox{\unitlength=0.5pt
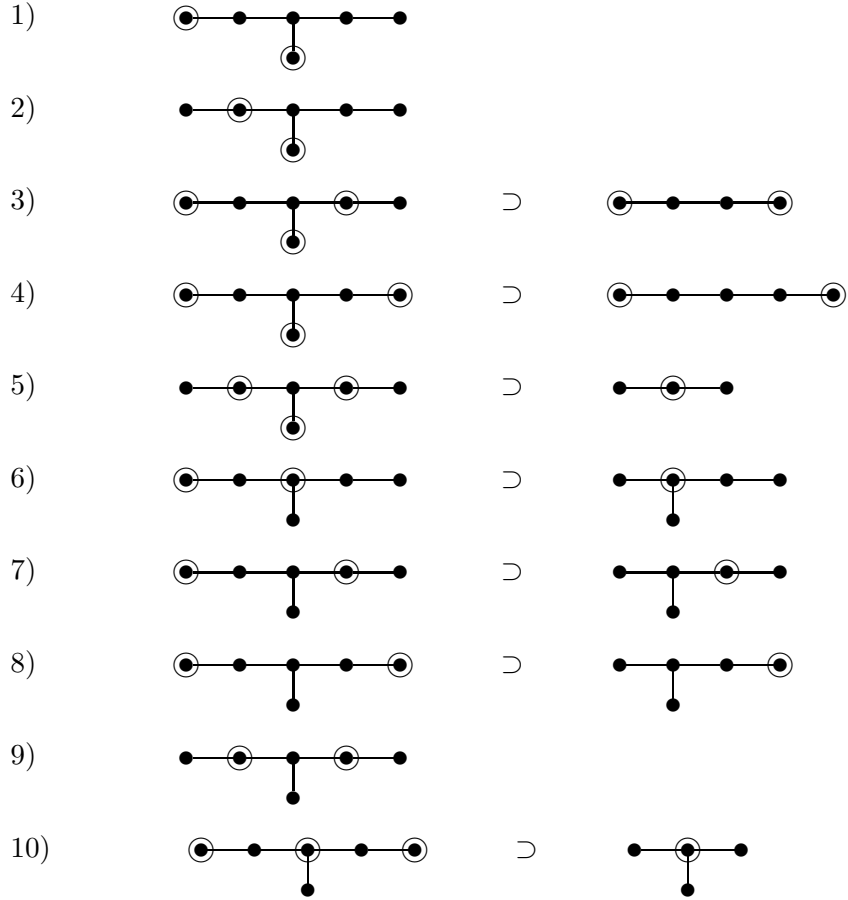
\begin{picture}(300,60)(-82,-5)
\put(10,0){\circle*{10}}
\put(10,0){\circle{18}}
\put(15,0){\line  (1,0){30}}
\put(50,0){\circle*{10}}
\put(55,0){\line  (1,0){30}}
\put(90,0){\circle*{10}}
\put(90,-5){\line  (0,-1){20}}
\put(90,-30){\circle*{10}}
\put(90,-30){\circle{18}}
\put(95,0){\line  (1,0){30}}
\put(130,0){\circle*{10}}
\put(135,0){\line  (1,0){30}}
\put(170,0){\circle*{10}}
  \end{picture}
  }
&&\\

 \end{array}$

$\begin{array}{clcl}
2)& \hbox{\unitlength=0.5pt
\begin{picture}(300,60)(-82,-5)
\put(10,0){\circle*{10}}
 
\put(15,0){\line  (1,0){30}}
\put(50,0){\circle*{10}}
\put(50,0){\circle{18}}
\put(55,0){\line  (1,0){30}}
\put(90,0){\circle*{10}}
\put(90,-5){\line  (0,-1){20}}
\put(90,-30){\circle*{10}}
\put(90,-30){\circle{18}}
\put(95,0){\line  (1,0){30}}
\put(130,0){\circle*{10}}
\put(135,0){\line  (1,0){30}}
\put(170,0){\circle*{10}}
  \end{picture}
  }
&&\\

 \end{array}$

$\begin{array}{clcl}
3)& \hbox{\unitlength=0.5pt
\begin{picture}(300,60)(-82,-5)
\put(10,0){\circle*{10}}
\put(10,0){\circle{18}} 
\put(15,0){\line  (1,0){30}}
\put(50,0){\circle*{10}}
 
\put(55,0){\line  (1,0){30}}
\put(90,0){\circle*{10}}
\put(90,-5){\line  (0,-1){20}}
\put(90,-30){\circle*{10}}
\put(90,-30){\circle{18}}
\put(95,0){\line  (1,0){30}}
\put(130,0){\circle*{10}}
\put(130,0){\circle{18}}
\put(135,0){\line  (1,0){30}}
\put(170,0){\circle*{10}}
  \end{picture}
  }
&\supset&
 \hskip -20pt\hbox{\unitlength=0.5pt
\begin{picture}(300,60)(-82,-5)
\put(10,0){\circle*{10}}
\put(10,0){\circle{18}} 
\put(15,0){\line  (1,0){30}}
\put(50,0){\circle*{10}}
 
\put(55,0){\line  (1,0){30}}
\put(90,0){\circle*{10}}

\put(95,0){\line  (1,0){30}}
\put(130,0){\circle*{10}}
\put(130,0){\circle{18}}
 
  \end{picture}
  }\\
  
   \end{array}$

$\begin{array}{clcl}
  4)& \hbox{\unitlength=0.5pt
\begin{picture}(300,60)(-82,-5)
\put(10,0){\circle*{10}}
\put(10,0){\circle{18}} 
\put(15,0){\line  (1,0){30}}
\put(50,0){\circle*{10}}
 
\put(55,0){\line  (1,0){30}}
\put(90,0){\circle*{10}}
\put(90,-5){\line  (0,-1){20}}
\put(90,-30){\circle*{10}}
\put(90,-30){\circle{18}}
\put(95,0){\line  (1,0){30}}
\put(130,0){\circle*{10}}
 
\put(135,0){\line  (1,0){30}}
\put(170,0){\circle*{10}}
\put(170,0){\circle{18}}
  \end{picture}
  }
&\supset&
 \hskip -20pt\hbox{\unitlength=0.5pt
\begin{picture}(300,60)(-82,-5)
\put(10,0){\circle*{10}}
 \put(10,0){\circle{18}}
\put(15,0){\line  (1,0){30}}
\put(50,0){\circle*{10}}
 
\put(55,0){\line  (1,0){30}}
\put(90,0){\circle*{10}}

\put(95,0){\line  (1,0){30}}
\put(130,0){\circle*{10}}
 
\put(135,0){\line  (1,0){30}}
\put(170,0){\circle*{10}}
\put(170,0){\circle{18}}
  \end{picture}
  }\\
  
   \end{array}$

$\begin{array}{clcl}
  5)& \hbox{\unitlength=0.5pt
\begin{picture}(300,60)(-82,-5)
\put(10,0){\circle*{10}}
 
\put(15,0){\line  (1,0){30}}
\put(50,0){\circle*{10}}
 \put(50,0){\circle{18}}
\put(55,0){\line  (1,0){30}}
\put(90,0){\circle*{10}}
\put(90,-5){\line  (0,-1){20}}
\put(90,-30){\circle*{10}}
\put(90,-30){\circle{18}}
\put(95,0){\line  (1,0){30}}
\put(130,0){\circle*{10}}
 \put(130,0){\circle{18}}
\put(135,0){\line  (1,0){30}}
\put(170,0){\circle*{10}}
 
  \end{picture}
  }
&\supset&
 \hskip -20pt\hbox{\unitlength=0.5pt
\begin{picture}(300,60)(-82,-5)
\put(10,0){\circle*{10}}
 
\put(15,0){\line  (1,0){30}}
\put(50,0){\circle*{10}}
 \put(50,0){\circle{18}}
\put(55,0){\line  (1,0){30}}
\put(90,0){\circle*{10}}

   \end{picture}
  }\\
   \end{array}$

$\begin{array}{clcl}
  
    6)& \hbox{\unitlength=0.5pt
\begin{picture}(300,60)(-82,-5)
\put(10,0){\circle*{10}}
 \put(10,0){\circle{18}}
\put(15,0){\line  (1,0){30}}
\put(50,0){\circle*{10}}
 
\put(55,0){\line  (1,0){30}}
\put(90,0){\circle*{10}}
\put(90,0){\circle{18}}
\put(90,-5){\line  (0,-1){20}}
\put(90,-30){\circle*{10}}
 
\put(95,0){\line  (1,0){30}}
\put(130,0){\circle*{10}}
 
\put(135,0){\line  (1,0){30}}
\put(170,0){\circle*{10}}
 
  \end{picture}
  }
&\supset&
 \hskip -20pt\hbox{\unitlength=0.5pt
\begin{picture}(300,60)(-82,-5)

\put(10,0){\circle*{10}}
 
\put(15,0){\line  (1,0){30}}
\put(50,0){\circle*{10}}
\put(50,0){\circle{18}}
\put(50,-5){\line  (0,-1){20}}
\put(50,-30){\circle*{10}}
 
\put(55,0){\line  (1,0){30}}
\put(90,0){\circle*{10}}
 
\put(95,0){\line  (1,0){30}}
\put(130,0){\circle*{10}}
 
  \end{picture}
  }
\\
 \end{array}$

$\begin{array}{clcl}

  7)& \hbox{\unitlength=0.5pt
\begin{picture}(300,60)(-82,-5)
\put(10,0){\circle*{10}}
 \put(10,0){\circle{18}}
\put(15,0){\line  (1,0){30}}
\put(50,0){\circle*{10}}
 
\put(55,0){\line  (1,0){30}}
\put(90,0){\circle*{10}}
 
\put(90,-5){\line  (0,-1){20}}
\put(90,-30){\circle*{10}}
 
\put(95,0){\line  (1,0){30}}
\put(130,0){\circle*{10}}
\put(130,0){\circle{18}} 
\put(135,0){\line  (1,0){30}}
\put(170,0){\circle*{10}}
 
  \end{picture}
  }
&\supset&
 \hskip -20pt\hbox{\unitlength=0.5pt
\begin{picture}(300,60)(-82,-5)

\put(10,0){\circle*{10}}
 
\put(15,0){\line  (1,0){30}}
\put(50,0){\circle*{10}}
 
\put(50,-5){\line  (0,-1){20}}
\put(50,-30){\circle*{10}}
 
\put(55,0){\line  (1,0){30}}
\put(90,0){\circle*{10}}
 \put(90,0){\circle{18}}
\put(95,0){\line  (1,0){30}}
\put(130,0){\circle*{10}}
 
  \end{picture}
  }
\\
 \end{array}$

$\begin{array}{clcl}

  8)& \hbox{\unitlength=0.5pt
\begin{picture}(300,60)(-82,-5)
\put(10,0){\circle*{10}}
 \put(10,0){\circle{18}}
\put(15,0){\line  (1,0){30}}
\put(50,0){\circle*{10}}
 
\put(55,0){\line  (1,0){30}}
\put(90,0){\circle*{10}}
 
\put(90,-5){\line  (0,-1){20}}
\put(90,-30){\circle*{10}}
 
\put(95,0){\line  (1,0){30}}
\put(130,0){\circle*{10}}
 
\put(135,0){\line  (1,0){30}}
\put(170,0){\circle*{10}}
\put(170,0){\circle{18}}
 
  \end{picture}
  }
&\supset&
 \hskip -20pt\hbox{\unitlength=0.5pt
\begin{picture}(300,60)(-82,-5)

\put(10,0){\circle*{10}}
 
\put(15,0){\line  (1,0){30}}
\put(50,0){\circle*{10}}
 
\put(50,-5){\line  (0,-1){20}}
\put(50,-30){\circle*{10}}
 
\put(55,0){\line  (1,0){30}}
\put(90,0){\circle*{10}}
 
\put(95,0){\line  (1,0){30}}
\put(130,0){\circle*{10}}
\put(130,0){\circle{18}}
 
  \end{picture}
  }
\\
 \end{array}$

$\begin{array}{clcl}
  9)& \hbox{\unitlength=0.5pt
\begin{picture}(300,60)(-82,-5)
\put(10,0){\circle*{10}}
 
\put(15,0){\line  (1,0){30}}
\put(50,0){\circle*{10}}
\put(50,0){\circle{18}}
 
\put(55,0){\line  (1,0){30}}
\put(90,0){\circle*{10}}
 
\put(90,-5){\line  (0,-1){20}}
\put(90,-30){\circle*{10}}
 
\put(95,0){\line  (1,0){30}}
\put(130,0){\circle*{10}}
\put(130,0){\circle{18}}
 
\put(135,0){\line  (1,0){30}}
\put(170,0){\circle*{10}}

  \end{picture}
  }
&&
 \\
 \end{array}$

$\begin{array}{clcl}
 10)& \hbox{\unitlength=0.5pt
\begin{picture}(300,60)(-82,-5)
\put(10,0){\circle*{10}}
 \put(10,0){\circle{18}}
\put(15,0){\line  (1,0){30}}
\put(50,0){\circle*{10}}
 
\put(55,0){\line  (1,0){30}}
\put(90,0){\circle*{10}}
\put(90,0){\circle{18}}
\put(90,-5){\line  (0,-1){20}}
\put(90,-30){\circle*{10}}
 
\put(95,0){\line  (1,0){30}}
\put(130,0){\circle*{10}}
 
\put(135,0){\line  (1,0){30}}
\put(170,0){\circle*{10}}
\put(170,0){\circle{18}}
 
  \end{picture}
  }
&\supset&
 \hskip -20pt\hbox{\unitlength=0.5pt
\begin{picture}(300,60)(-82,-5)

\put(10,0){\circle*{10}}
 
\put(15,0){\line  (1,0){30}}
\put(50,0){\circle*{10}}
\put(50,0){\circle{18}} 
\put(50,-5){\line  (0,-1){20}}
\put(50,-30){\circle*{10}}
 
\put(55,0){\line  (1,0){30}}
\put(90,0){\circle*{10}}

  \end{picture}
  }
\\

\end{array}$
\vskip 20pt
Let us consider the case $1)$ in the list above. The corresponding $PV$ is infinitesimally equivalent to $(G,V)$ where $G=GL(5)\times {\bb C}^*$, $V=M(5,1)\oplus Skew(5)$ and the action is given by:
$(g,a)(X,Y)=(agX,gY\,^t\hskip -2pt g)$  where $\, a\in {\bb C}^*, g\in GL(5), X\in M(5,1), Y\in Skew(5)$.

Define $J=\left[\hskip -5pt\begin{array}{cc}0&I_{2}\\
  -I_{2}&0
  \end{array}\hskip -5pt\right]\in Skew(4)$. Let then 
  $$X_{0}=\left[\hskip -5pt\begin{array}{c}0\\
  1
  \end{array}\hskip -5pt\right] \in M(5,1) \text{ and } Y_{0}=\left[\hskip -5pt\begin{array}{cc}J&0\\
  0&0
  \end{array}\hskip -5pt\right]\in Skew(5).$$
  An easy computation shows that $(X_{0},Y_{0})$ is generic and that its isotropy subgroup is the set of pairs of matrices of the form
  $(\left[\hskip -5pt\begin{array}{cc}A&0\\
  0&a
  \end{array}\hskip -5pt\right], a^{-1})$, where $A\in Sp(2), a\in {\bb C}^*$. Hence the $PV$ is regular and one easily shows that the unique fundamental relative invariant is given by    
  $$f(X,Y)=Pf(\left[\hskip -5pt\begin{array}{cc}Y&X\\
  -(^t\hskip -2pt X)&0
  \end{array}\hskip -5pt\right]).$$
  
  For the cases $2)$ and $9)$ one computes a generic isotropy isotropy subgroup and one observes that it is not reductive.

 \end{proof}
 
 A consequence of the preceding classification is the following statement.
 
 \begin{prop}\label{prop-Q=1} The $Q$-irreducible $PV$'s of parabolic type   are $1$-irreducible. In other words the three definitions of irreducibility given in Definition \ref{def-irreductibles} are equivalent for $PV$'s of parabolic type. 
 
  \end{prop}
 
\begin{rem} The exceptional $Q$-irreducible $PV$'s arising in $E_{6}$, $E_{7}$ and $E_{8}$ are particular cases of families of $Q$-irreducible $PV$'s which are not parabolic in general. More precisely the representations
 $$(GL(n)\times {\bb C}^*, [\Lambda_{1}(GL(n))\otimes \Box]\oplus [\Lambda_{2}(GL(n))\otimes Id]) (n \text{ odd})  $$
 and
$$(GL(n)\times GL(n-1), [\Lambda_{1}(GL(n))\otimes \Lambda_{1}(GL(n-1))]\oplus [\Lambda_{1}(GL(n))\otimes Id])$$
are $1$-irreducible $PV$'s. (Here $\Box$ denotes the one dimensional representation of ${\bb C}^*$ on ${\bb C}$ by multiplications).
The first representation is an extension of the $E_{6}$ and $E_{8}$ cases, the second one is an extension of the $E_{7}$ case. 

For  the first representation the fundamental relative invariant is given by 
$$f(X,Y)=Pf(\left[\hskip -5pt\begin{array}{cc}Y&X\\
  -(^t\hskip -2pt X)&0
  \end{array}\hskip -5pt\right]), (X\in M(n,1), Y\in Skew(n))$$
  and for the second representation  it is given by
  $$f(X,Y)=\det(\left[\hskip -5pt\begin{array}{cc}X&Y\\
   
  \end{array}\hskip -5pt\right]), (X\in M(n,n-1), Y\in M(n,1)).$$
  
  Note that the first $PV$ above is example 8) p. 95 of \cite{Kimura-2} and that the second $PV$ is 
  
  \begin{rem}\footnote{I would like to thank Tatsuo Kimura for providing me  this example.}
 Let $G= (GL(n))^{p+1} $, and let $V=(M(n))^p$. We denote by $g=(g_{i})$ an element in $G$ and by $v=(v_{j})$ an element in $V$. Consider the representation of $G$ on $V$ defined by $(g_{i}).(v_{j})= (g_{j}v_{j}g_{j+1}^{-1})$. Then $(G,V)$ is a regular $PV$ with $p$ fundamental relative invariants given by $f_{j}(v)=\det(v_{j})$. Let $N= pn^2=\dim(V)$ and consider the castling transformation (see \cite{Sato-Kimura} or \cite{Kimura}) of $(G,V)$ given by  $(G\times GL(N-1),V\otimes \C^{N-1})$. It is known (\cite{Sato-Kimura} p.67-68, and Remark 26 p. 73), that the regularity and the number of fundamental relative invariants does not change under castling transformation, therefore $(G\times GL(N-1), V\otimes \C^{N-1})$ is regular and has also $p$ fundamental relative invariants. But it is easy to see that any proper  $ G\times GL(N-1)$-invariant subspace of $V\otimes \C^{N-1}$ is of the form $U\otimes \C^{N-1}$, with $\dim(U)< N-1$, and then $(G\times GL(N-1),U\otimes \C^{N-1})$ is a so-called  {\it trivial} $PV$, which has no fundamental relative invariant. Hence  $(G\times GL(N-1),V\otimes \C^{N-1})$ is a $Q$-irreducible $PV$ which is not $1$-irreducible. Therefore Proposition \ref{prop-Q=1} is no longer true for non parabolic $PV$'s.
  
  \end{rem}

  \begin{rem} We have proved in \cite{rub-note-PV} that an irreducible $PV$ of parabolic type is regular if and only if the corresponding grading element $H_{\theta}$ (see section 4.1) is the semi-simple element of an ${\go {sl}}_{2}$-triple. As the weighted Dynkin diagrams of type   $E_{6}$, $E_{7}$,  $E_{8}$  in Table 1 below do not appear  in tables 18, 19, 20 of \cite{Dynkin}, such a result is no longer true for $Q$-irreducible $PV$'s of parabolic type.
  
  \end{rem}

\end{rem}

\vfill\eject

\begin{tabular} {|c|c|}
\hline
\multicolumn{2}{|c|}{}\\
\multicolumn{2}{|c|}{ \rm Table 1: non irreducible, $Q$-irreducible $PV$'s of parabolic type} \\
\multicolumn{2}{|c|}{}\\
 \hline
 {$\begin{array}{c}\\ A_{n}\\ (p_{2}>p_{1}\geq0)\end{array}$}&\hskip -40pt \hbox to 9cm{\unitlength=0.5pt
  \begin{picture}(550,-20)
 \put(90,10){\circle*{10}}
 \put(95,10){\line (1,0){30}}
\put(130,10){\circle*{10}}
 
\put(140,10){\circle*{1}}
\put(145,10){\circle*{1}}
\put(150,10){\circle*{1}}
\put(155,10){\circle*{1}}
\put(160,10){\circle*{1}}
\put(165,10){\circle*{1}}
\put(170,10){\circle*{1}}
\put(175,10){\circle*{1}}
\put(180,10){\circle*{1}}
 \put(190,10){\circle*{10}}
 \put(155,-10){$p_{1}$}
 
\put(195,10){\line (1,0){30}}
\put(235,10){\circle*{10}}
 
\put(235,10){\circle{18}}
\put(245,10){\line (1,0){30}}
\put(280,10){\circle*{10}}
 
\put(290,10){\circle*{1}}
\put(295,10){\circle*{1}}
\put(300,10){\circle*{1}}
\put(305,10){\circle*{1}}
\put(310,10){\circle*{1}}
\put(315,10){\circle*{1}}
\put(320,10){\circle*{1}}
\put(325,10){\circle*{1}}
\put(330,10){\circle*{1}}
\put(305,-10){$p_{2}$}
\put(340,10){\circle*{10}}
 
\put(345,10){\line (1,0){30}}
\put(385,10){\circle*{10}}
\put(385,10){\circle{18}}
\put(395,10){\line (1,0){30}}
\put(430,10){\circle*{10}}

\put(435,10){\circle*{1}}
\put(440,10){\circle*{1}}
\put(445,10){\circle*{1}}
\put(450,10){\circle*{1}}
\put(455,10){\circle*{1}}
\put(460,10){\circle*{1}}
\put(465,10){\circle*{1}}
\put(470,10){\circle*{1}}
\put(475,10){\circle*{1}}
\put(485,10){\circle*{10}}
\put(490,10){\line (1,0){30}}
\put(525,10){\circle*{10}}
 \put(450,-10){$p_{1}$}
  \end{picture}
}\\
\hline
 {$\begin{array}{c}\\B_{n}\\ p_{2}>p_{1},\\ 2p_{3}=p_{1}, p_{3}\geq0)\end{array}$}&\hskip -40pt \hbox to 9,5cm {\unitlength=0.5pt
  \begin{picture}(550,-20)
 \put(90,10){\circle*{10}}
 \put(95,10){\line (1,0){30}}
\put(130,10){\circle*{10}}
 
\put(140,10){\circle*{1}}
\put(145,10){\circle*{1}}
\put(150,10){\circle*{1}}
\put(155,10){\circle*{1}}
\put(160,10){\circle*{1}}
\put(165,10){\circle*{1}}
\put(170,10){\circle*{1}}
\put(175,10){\circle*{1}}
\put(180,10){\circle*{1}}
 \put(190,10){\circle*{10}}
 \put(155,-10){$p_{1}$}
 
\put(195,10){\line (1,0){30}}
\put(235,10){\circle*{10}}
 
\put(235,10){\circle{18}}
\put(245,10){\line (1,0){30}}
\put(280,10){\circle*{10}}
 
\put(290,10){\circle*{1}}
\put(295,10){\circle*{1}}
\put(300,10){\circle*{1}}
\put(305,10){\circle*{1}}
\put(310,10){\circle*{1}}
\put(315,10){\circle*{1}}
\put(320,10){\circle*{1}}
\put(325,10){\circle*{1}}
\put(330,10){\circle*{1}}
\put(305,-10){$p_{2}$}
\put(340,10){\circle*{10}}
 
\put(345,10){\line (1,0){30}}
\put(385,10){\circle*{10}}
\put(385,10){\circle{18}}
\put(395,10){\line (1,0){30}}
\put(430,10){\circle*{10}}

\put(435,10){\circle*{1}}
\put(440,10){\circle*{1}}
\put(445,10){\circle*{1}}
\put(450,10){\circle*{1}}
\put(455,10){\circle*{1}}
\put(460,10){\circle*{1}}
\put(465,10){\circle*{1}}
\put(470,10){\circle*{1}}
\put(475,10){\circle*{1}}
\put(485,10){\circle*{10}}
\put(484,12){\line (1,0){41}}
\put(484,8){\line (1,0){41}}
\put(500,5.5){$>$}
\put(530,10){\circle*{10}}
 
 \put(450,-10){$p_{3}$}
 
 \end{picture}
}\\
\hline
 {$\begin{array}{c}\\C_{n}\\ p_{2}>p_{1},\\ 2p_{3}=p_{1}+1, p_{3}>0\\
 p_{2}\text{ odd}\end{array}$}& \hskip -40pt\hbox to 9,5cm {\unitlength=0.5pt
  \begin{picture}(550,-20)
 \put(90,10){\circle*{10}}
 \put(95,10){\line (1,0){30}}
\put(130,10){\circle*{10}}
 
\put(140,10){\circle*{1}}
\put(145,10){\circle*{1}}
\put(150,10){\circle*{1}}
\put(155,10){\circle*{1}}
\put(160,10){\circle*{1}}
\put(165,10){\circle*{1}}
\put(170,10){\circle*{1}}
\put(175,10){\circle*{1}}
\put(180,10){\circle*{1}}
 \put(190,10){\circle*{10}}
 \put(155,-10){$p_{1}$}
 
\put(195,10){\line (1,0){30}}
\put(235,10){\circle*{10}}
 
\put(235,10){\circle{18}}
\put(245,10){\line (1,0){30}}
\put(280,10){\circle*{10}}
 
\put(290,10){\circle*{1}}
\put(295,10){\circle*{1}}
\put(300,10){\circle*{1}}
\put(305,10){\circle*{1}}
\put(310,10){\circle*{1}}
\put(315,10){\circle*{1}}
\put(320,10){\circle*{1}}
\put(325,10){\circle*{1}}
\put(330,10){\circle*{1}}
\put(305,-10){$p_{2}$}
\put(340,10){\circle*{10}}
 
\put(345,10){\line (1,0){30}}
\put(385,10){\circle*{10}}
\put(385,10){\circle{18}}
\put(395,10){\line (1,0){30}}
\put(430,10){\circle*{10}}

\put(435,10){\circle*{1}}
\put(440,10){\circle*{1}}
\put(445,10){\circle*{1}}
\put(450,10){\circle*{1}}
\put(455,10){\circle*{1}}
\put(460,10){\circle*{1}}
\put(465,10){\circle*{1}}
\put(470,10){\circle*{1}}
\put(475,10){\circle*{1}}
\put(485,10){\circle*{10}}
\put(484,12){\line (1,0){41}}
\put(484,8){\line (1,0){41}}
\put(500,5.5){$<$}
\put(530,10){\circle*{10}}
 
 \put(450,-10){$p_{3}$}
 
 \end{picture}
}\\
\hline
  {$\begin{array}{c}\\D_{n}^1\\ p_{2}>p_{1},\\ 2p_{3}=p_{1}+1, p_{3}\geq 2\\
 p_{2}\text{ even}\end{array}$}& \hskip -40pt\hbox to 9cm {\unitlength=0.5pt
  \begin{picture}(550,-20)
 \put(90,10){\circle*{10}}
 \put(95,10){\line (1,0){30}}
\put(130,10){\circle*{10}}
 
\put(140,10){\circle*{1}}
\put(145,10){\circle*{1}}
\put(150,10){\circle*{1}}
\put(155,10){\circle*{1}}
\put(160,10){\circle*{1}}
\put(165,10){\circle*{1}}
\put(170,10){\circle*{1}}
\put(175,10){\circle*{1}}
\put(180,10){\circle*{1}}
 \put(190,10){\circle*{10}}
 \put(155,-10){$p_{1}$}
 
\put(195,10){\line (1,0){30}}
\put(235,10){\circle*{10}}
 
\put(235,10){\circle{18}}
\put(245,10){\line (1,0){30}}
\put(280,10){\circle*{10}}
 
\put(290,10){\circle*{1}}
\put(295,10){\circle*{1}}
\put(300,10){\circle*{1}}
\put(305,10){\circle*{1}}
\put(310,10){\circle*{1}}
\put(315,10){\circle*{1}}
\put(320,10){\circle*{1}}
\put(325,10){\circle*{1}}
\put(330,10){\circle*{1}}
\put(305,-10){$p_{2}$}
\put(340,10){\circle*{10}}
 
\put(345,10){\line (1,0){30}}
\put(385,10){\circle*{10}}
\put(385,10){\circle{18}}
\put(395,10){\line (1,0){30}}
\put(430,10){\circle*{10}}

\put(435,10){\circle*{1}}
\put(440,10){\circle*{1}}
\put(445,10){\circle*{1}}
\put(450,10){\circle*{1}}
\put(455,10){\circle*{1}}
\put(460,10){\circle*{1}}
\put(465,10){\circle*{1}}
\put(470,10){\circle*{1}}
\put(475,10){\circle*{1}}
\put(485,10){\circle*{10}}
\put(490,14){\line (1,1){20}}
\put(515,36){\circle*{10}}
\put(490,6){\line (1,-1){20}}
\put(515,-16){\circle*{10}}

 \put(450,-10){$p_{3}$}
 
 \end{picture}
 }\\
 \hline
  {$\begin{array}{c}\\D_{n}^2\\ p_{2}\geq 2,\\ p_{1}=p_{2}-1, p_{2} \text{ even}\end{array}$}& \hskip -40pt\hbox to 9cm {\unitlength=0.5pt
  \begin{picture}(550,-20)
 \put(90,10){\circle*{10}}
 \put(95,10){\line (1,0){30}}
\put(130,10){\circle*{10}}
 
\put(140,10){\circle*{1}}
\put(145,10){\circle*{1}}
\put(150,10){\circle*{1}}
\put(155,10){\circle*{1}}
\put(160,10){\circle*{1}}
\put(165,10){\circle*{1}}
\put(170,10){\circle*{1}}
\put(175,10){\circle*{1}}
\put(180,10){\circle*{1}}
 \put(190,10){\circle*{10}}
 \put(155,-10){$p_{1}$}
 
\put(195,10){\line (1,0){30}}
\put(235,10){\circle*{10}}
 
\put(235,10){\circle{18}}
\put(245,10){\line (1,0){30}}
\put(280,10){\circle*{10}}
 
\put(290,10){\circle*{1}}
\put(295,10){\circle*{1}}
\put(300,10){\circle*{1}}
\put(305,10){\circle*{1}}
\put(310,10){\circle*{1}}
\put(315,10){\circle*{1}}
\put(320,10){\circle*{1}}
\put(325,10){\circle*{1}}
\put(330,10){\circle*{1}}
\put(380,-10){$p_{2}$}
\put(340,10){\circle*{10}}
 
\put(345,10){\line (1,0){34}}
\put(385,10){\circle*{10}}
 
\put(390,10){\line (1,0){34}}
\put(430,10){\circle*{10}}

\put(435,10){\circle*{1}}
\put(440,10){\circle*{1}}
\put(445,10){\circle*{1}}
\put(450,10){\circle*{1}}
\put(455,10){\circle*{1}}
\put(460,10){\circle*{1}}
\put(465,10){\circle*{1}}
\put(470,10){\circle*{1}}
\put(475,10){\circle*{1}}
\put(485,10){\circle*{10}}
\put(490,14){\line (1,1){20}}
\put(515,36){\circle*{10}}
\put(515,36){\circle{18}}
\put(490,6){\line (1,-1){20}}
\put(515,-16){\circle*{10}}
\end{picture}
}\\
\hline
  {$ \begin{array}{c}
\\ D_{n}^3\\ p_{2}> 1\\
\\
 \end{array}$}& \hskip -40pt \hbox to 9cm {\unitlength=0.5pt
\begin{picture}(550,-20)
  \put(90,10){\circle*{10}}
 \put(95,10){\line (1,0){30}}
\put(133,10){\circle*{10}}
\put(133,10){\circle{18}}
 \put(140,10){\line (1,0){40}}
  \put(185,10){\circle*{10}}
\put(190,10){\line (1,0){40}}
\put(235,10){\circle*{10}}
\put(240,10){\line (1,0){35}}
\put(280,10){\circle*{10}}
\put(290,10){\circle*{1}}
\put(295,10){\circle*{1}}
\put(300,10){\circle*{1}}
\put(305,10){\circle*{1}}
\put(310,10){\circle*{1}}
\put(315,10){\circle*{1}}
\put(320,10){\circle*{1}}
\put(325,10){\circle*{1}}
\put(330,10){\circle*{1}}
\put(350,-10){$p_{2}$}
\put(340,10){\circle*{10}}
\put(345,10){\line (1,0){34}}
\put(385,10){\circle*{10}}
\put(390,10){\line (1,0){34}}
\put(430,10){\circle*{10}}
\put(435,10){\circle*{1}}
\put(440,10){\circle*{1}}
\put(445,10){\circle*{1}}
\put(450,10){\circle*{1}}
\put(455,10){\circle*{1}}
\put(460,10){\circle*{1}}
\put(465,10){\circle*{1}}
\put(470,10){\circle*{1}}
\put(475,10){\circle*{1}}
\put(485,10){\circle*{10}}
\put(490,14){\line (1,1){20}}
\put(515,36){\circle*{10}}
\put(515,36){\circle{18}}
\put(490,6){\line (1,-1){20}}
\put(515,-16){\circle*{10}}
\put(515,-16){\circle{18}}
\end{picture}
}\\
\hline
$\begin{array}{c}E_{6}\\
{}
\end{array}$& \hbox{\unitlength=0.5pt
\begin{picture}(300,60)(-82,-43)
\put(10,0){\circle*{10}}
\put(10,0){\circle{18}}
\put(15,0){\line  (1,0){30}}
\put(50,0){\circle*{10}}
\put(55,0){\line  (1,0){30}}
\put(90,0){\circle*{10}}
\put(90,-5){\line  (0,-1){30}}
\put(90,-40){\circle*{10}}
\put(90,-40){\circle{18}}
\put(95,0){\line  (1,0){30}}
\put(130,0){\circle*{10}}
\put(135,0){\line  (1,0){30}}
\put(170,0){\circle*{10}}
  
\end{picture}
}
 \\
 \hline
 $\begin{array}{c}E_{7}\\
{}
\end{array}$& \hbox{\unitlength=0.5pt
\begin{picture}(300,60)(-82,-43)
\put(10,0){\circle*{10}}
 
\put(15,0){\line  (1,0){30}}
\put(50,0){\circle*{10}}
\put(55,0){\line  (1,0){30}}
\put(90,0){\circle*{10}}
\put(90,-5){\line  (0,-1){30}}
\put(90,-40){\circle*{10}}
\put(90,-40){\circle{18}}
\put(95,0){\line  (1,0){30}}
\put(130,0){\circle*{10}}
\put(130,0){\circle{18}}
\put(135,0){\line  (1,0){30}}
\put(170,0){\circle*{10}}
\put(175,0){\line  (1,0){30}}
\put(210,0){\circle*{10}}

\end{picture}
}
 \\
 \hline
  $\begin{array}{c}E_{8}\\
{}
\end{array}$& \hbox{\unitlength=0.5pt
\begin{picture}(300,60)(-82,-43)
\put(10,0){\circle*{10}}
\put(10,0){\circle{18}}

\put(15,0){\line  (1,0){30}}
\put(50,0){\circle*{10}}
\put(55,0){\line  (1,0){30}}
\put(90,0){\circle*{10}}
\put(90,-5){\line  (0,-1){30}}
\put(90,-40){\circle*{10}}
\put(90,-40){\circle{18}}
\put(95,0){\line  (1,0){30}}
\put(130,0){\circle*{10}}
 
\put(135,0){\line  (1,0){30}}
\put(170,0){\circle*{10}}
\put(175,0){\line  (1,0){30}}
\put(210,0){\circle*{10}}
\put(215,0){\line  (1,0){30}}
\put(250,0){\circle*{10}}

\end{picture}
}
 \\
 \hline

\end{tabular}


\end{document}